\documentclass[a4paper,twoside,11pt]{article}
\usepackage{a4wide}
\usepackage[utf8]{inputenc}
\usepackage[english]{babel}
\usepackage{color} 
\usepackage{amsmath,amscd,amssymb,amsthm,amsfonts}
\usepackage{xparse}
\usepackage{stmaryrd}
\usepackage{bbm}
\usepackage{graphicx}
\usepackage{hyperref}
\usepackage{enumitem}

\theoremstyle{plain}
\newtheorem{theorem}{Theorem}
\newtheorem*{theorem*}{Theorem}
\newtheorem{lemma}{Lemma}[section]
\newtheorem{corollary}[lemma]{Corollary}
\newtheorem{proposition}[lemma]{Proposition}
\newtheorem{fact}[lemma]{Fact}

\theoremstyle{definition}

\theoremstyle{remark}
\newtheorem{definition}[lemma]{Definition}
\newtheorem{remark}[lemma]{Remark}
\newtheorem{example}[lemma]{Example}

\numberwithin{equation}{section}

\def\N{\ensuremath{\mathbb{N}}}

\def\R{\ensuremath{\mathbb{R}}}

\def\ep{\varepsilon}

\def\E{\ensuremath{\mathbf{E}}}

\def\P{\ensuremath{\mathbf{P}}}

\def\Ind{\ensuremath{\mathbbm{1}}}

\def\to{\rightarrow}

\def\tand{\ensuremath{\text{ and }}}

\def\tif{\ensuremath{\text{ if }}}

\def\rT{\mathcal{T}}
\def\rV{\mathcal{V}}
\def\rW{\mathcal{W}}
\def\ind{\ensuremath{\mathbbm{1}}}

\def\Trfin{\mathfrak T_f}
\def\Trinf{\mathfrak T_e}
\def\Trinfall{\mathfrak T}
\def\Trone{\mathfrak T_1}
\def\Trlfin{\Trfin^\ell}
\def\Trlinf{\Trinf^\ell}
\def\Trlinfall{\Trinfall^\ell}
\def\Trlone{\Trone^\ell}

\def\Tdfin{\mathbb T_f}
\def\Tdinf{\mathbb T_e}
\def\Tdinfall{\mathbb T}

\def\Sp{\mathrm{Spine}}
\def\S{\Theta}
\def\root{\rho}

\def\Tdspine{\overline{\mathbb T}}
\def\Trspine{\overline{\mathfrak T}^\ell}
\def\PSpace{\mathcal P}

\DeclareMathOperator\Poi{Poi}

\DeclareDocumentCommand \Sop { O{p,q} }
{
\mathcal C_{#1}
}

\DeclareDocumentCommand \Srop { O{p,q} }
{
\mathcal S_{#1}
}

\DeclareDocumentCommand \Cop { m m }
{
\mathcal C(#1,#2)
}

\DeclareDocumentCommand \Dop { o o }
{
\IfNoValueTF {#1}
{
\mathcal D
}
{
\IfNoValueTF {#2}
{\mathcal D(#1)}
{\mathcal D(#1,#2)}
}
}

\DeclareDocumentCommand \M { m }
{
\mathcal M_1(#1)
}

\DeclareDocumentCommand \Mf { m }
{
\mathcal M_f(#1)
}

\DeclareDocumentCommand \tembed { m }
{
\iota(#1)
}

\DeclareDocumentCommand \TODO { o }%
{%
\IfNoValueTF {#1}%
{{\color{red}TODO}}%
{{\color{red}TODO:#1}}%
}

\DeclareDocumentCommand \EPO { o }%
{%
\IfNoValueTF {#1}%
{\vartriangleleft}%
{\vartriangleleft_{#1}}%
}

\makeatletter
\renewcommand*{\@fnsymbol}[1]{\ensuremath{\ifcase#1\or *\or **\or \dagger\or \ddagger\or
   \mathsection\or \mathparagraph\or \|\or \dagger\dagger
   \or \ddagger\ddagger \else\@ctrerr\fi}}
\makeatother

\author{\textsc{Olivier H\'enard} 
\thanks{Laboratoire de Math\'ematiques d'Orsay, Univ. Paris-Sud, CNRS, Universit\'e Paris-Saclay, 91405 Orsay Cedex, France. Supported by EPSRC grant EP/J004022/2 at the time of this work.} 
\and
\textsc{Pascal Maillard}\thanks{Laboratoire de Math\'ematiques d'Orsay, Univ. Paris-Sud, CNRS, Universit\'e Paris-Saclay, 91405 Orsay Cedex, France. Partially supported by a grant from the Israel Science Foundation at the time this article was written.}
}
\title{On trees invariant under edge contraction}
\begin{document}

\maketitle

\begin{abstract}
We study random trees which are invariant in law under 
the operation of contracting each edge independently with probability $p\in(0,1)$.  
We show that all such trees can be constructed through Poissonian sampling  from a certain class of random
measured $\R$-trees satisfying a natural scale invariance property. This has connections to exchangeable 
partially ordered sets, real-valued self-similar increasing processes and quasi-stationary distributions 
of Galton--Watson processes.
\end{abstract}

\section{Introduction}

Take a random rooted tree $T$ (in the graph sense) and contract each edge independently with 
probability $p\in(0,1)$. Are there (necessarily infinite) random trees which are invariant in law under this operation? 
Trivial examples are the (semi)-infinite ray, i.e.\ the one-dimensional half-lattice $\N:= \{0,1,\ldots\}$ and with root $0$, 
or a (possibly random) number of copies of this graph whose roots are identified. 
Slightly less trivial examples are the previous graph with a bouquet of edges attached to each vertex, 
the number of edges in each bouquet being independent and identically distributed (iid) according to a geometric distribution (starting at $0$). 
There are however many more examples of such trees. Indeed, the following is the main result of this article, which is an informal statement of Theorem~\ref{th:self_similar}:
\begin{theorem*}
There is a one-to-one correspondence between random trees invariant under the above contraction operation and a certain class of continuum random trees invariant under rescaling by the factor $p$.
\end{theorem*}
The proof of this result will involve quite an extensive machinery of tools, including, for instance,
Gromov--Hausdorff--Prokhorov convergence of locally compact metric spaces 
and random exchangeable partial orders.
We furthermore give in this article constructions of several examples of such trees and completely characterize the class of the trees which are also invariant (in law) under translations along the spine. This has connections with real-valued self-similar increasing processes and with quasi-stationary distributions of subcritical Galton--Watson processes. 

The study of this problem originates in a geometrically motivated question asked to us by I.~Benjamini: 
Is it possible to find a law on the space of trees and a suitable renormalization procedure, 
i.e.\ a suitable random coarse-graining operation that preserves the tree structure, such that the 
law of the tree is invariant under this operation? 
The contraction operation is an example of such a renormalization procedure, 
which has the pleasant feature of allowing to characterize all 
locally finite random trees it leaves invariant. 

We are not aware of any similar results in the random tree literature. 
The operations that are usually considered on trees, for example random growth or pruning, 
always operate on leaves or whole subtrees instead of single interior vertices, see \cite{RE85, AP98}.
There might be a good reason for that: The trees we obtain are indeed very different from usual trees in the sense that they are very elongated, with long chains of vertices of degree 2, to which might be attached some bouquets of edges as in the above example. In particular, neither exponentially growing trees nor critical Galton--Watson trees (conditioned on non-extinction) are amongst them.

\subsection*{Definitions and statements of the results}

The precise statements of our results require a fair amount of definitions which we give in this section, occasionally referring to the appendix for details. Very impatient readers might directly jump to the statement of the main theorem (Theorem~\ref{th:self_similar}) and work their way backwards to understand all the definitions.

In this article, we consider rooted, locally finite trees $T=(V,E,\root)$ in the graph-theoretic sense; 
in combinatorics these are also known as \emph{unlabelled}, \emph{unorderered}, \emph{non-plane} or \emph{P\'olya} trees 
\cite{Stanley,FS2009,Drmota}.
We say that two such trees are \emph{equivalent} if there exists a root-preserving graph isomorphism between them and denote by $\Tdinfall$ the space of equivalence classes of trees. We usually identify an equivalence class with its representatives.
A path in $T$ is a finite or infinite sequence of adjacent, pairwise distinct vertices of $T$, and an end is an infinite path starting at the root.
The \emph{spine} of $T$, denoted by $\Sp(T)$, is the union of the ends of $T$, or $\{\rho\}$ in case $T$ has no end.
We denote by $\Tdinf\subset \Tdinfall$ the subspace of trees with finitely many ends and by $\Tdfin \subset \Tdinf$ the subspace of the finite trees. 

We endow the space $\Tdinfall$ (and its subspaces) with the topology of local convergence defined as follows: For a tree $T\in\Tdinfall$ and $k\in\N$, denote by $T^{\le k}$ its restriction to the vertices at (graph) distance at most $k$ from the root. A sequence $(T_n)_{n\ge0}$ in $\Tdinfall$ then is said to converge to $T\in\Tdinfall$ if for every $k\in\N$, $(T_n^{\le k})_{n\ge0}$ converges to $T^{\le k}$ in $\Tdfin$. This topology is metrizable in such a way that the space $\Tdinfall$ is a complete separable metric space, but not the space $\Tdinf$ (see Section~\ref{sec:Dtrees}). Note also that the induced topology on $\Tdfin$ is indeed the discrete topology.

In order to deal with \emph{random} trees, we endow the space $\Tdinfall$ with the Borel-$\sigma$-field induced by its topology. We then denote by $\M{\Tdinfall}$ the space of probability measures on $\Tdinfall$, endowed with the topology of weak convergence. 
We will often denote the elements of $\M{\Tdinfall}$ by $T$ as well, and will refer to them as \emph{random} trees. We similarly define $\M{\Tdinf}$ and $\M{\Tdfin}$. 

A rooted tree $T=(V,E,\root)\in\Tdinfall$ determines a partial order $\preceq_T$ on its vertex set $V$ by $v \preceq_T w$ if and only if $v$ lies on the (unique) path from $\root$ to $w$. In this case, we say that $v$ is an \emph{ancestor} of $w$, or that $w$ is a \emph{descendant} of $v$; accordingly, we call $\prec_T$ the \emph{ancestral relation} of the tree $T$. 
We also write $v \prec_T w$ when $v \preceq_T w$ and $v \neq w$.
Note that the tree $T$ can in fact be completely recovered from $\preceq_T$ and $\prec_T$\footnote{In fact, the tree can be viewed as the transitive reduction of the (acyclic) directed graph $(V,\prec_T)$.}.
This allows us to formally define a generic contraction operation as follows:

\begin{definition}
\label{def:contraction}
Let $T=(V,E,\root)\in \Tdinf$ and let $V'\subset V$ be a subset of its vertices containing the root and containing an infinite number of vertices on each infinite path. The \emph{contracted tree} $\Cop{T}{V'}$ is defined to be the 
rooted tree with vertex set $V'$, root $\root$ and whose partial order $\preceq_{\Cop{T}{V'}}$ is the restriction of $\preceq_T$ to $V'$. It is easy to see that $\Cop{T}{V'}$ is indeed locally finite and has a finite number of ends only, i.e. $\Cop{T}{V'}\in\Tdinf$.
\end{definition}

We then define the randomized contraction operation we will consider.
\begin{definition}
\label{def:contraction_pq}
Let $p,q\in(0,1)$. Let $T=(V,E,\rho)\in\Tdinf$. Set $V_0 = V\backslash\Sp(T)$ and $V_1 = \Sp(T)\backslash\root$. The random tree $\Sop(T)\in\M{\Tdinf}$ is defined to be equal to $\Cop{T}{V'}$, where $V'$ is the random subset\footnote{It is easy to show that this subset contains an infinite number of vertices on each infinite path.} of vertices containing $\root$, every vertex in $V_0$ independently with probability $p$ and every vertex in $V_1$ independently with probability $q$. 
\end{definition}
Note that if $p\ne q$, then the map $\Sop:\Tdinf\to\M{\Tdinf}$ is not continuous, because when a sequence of trees $T_n$ converges to a tree $T$, non-spine vertices can become spine vertices in the limit. However, if we define for $M\in\N$ the map $\Sop^M$, in which a vertex is kept in $V'$ with probability $q$ if it has a descendant at distance $M$, and with probability $p$ otherwise, then this map is easily seen to be continuous. In particular\footnote{See Section~\ref{sec:extension}.}, it extends to a continuous map $\Sop^M:\M{\Tdinfall}\to\M{\Tdinfall}$. Since $\Sop = \lim_{M\to\infty} \Sop^M$, we conclude that the map $\Sop$ is measurable and extends to a (measurable) map $\Sop:\M{\Tdinf}\to\M{\Tdinf}$. 
This allows us to write $\Sop(T)$ for a \emph{random} tree $T\in\M{\Tdinf}$.

\begin{definition}
\label{def:self_similar_discrete}
Let $p,q\in(0,1)$. We say that a random tree $T\in\M{\Tdinf}$ is \emph{$(p,q)$-self-similar}, if $T \stackrel{\text{law}}{=}\Sop(T)$.
\end{definition}
\begin{remark}
In defining the contraction operation $\Sop$, we restricted our discussion to trees in $\Tdinf$. 
This brings no restriction when studying locally finite self-similar random  trees, since
there are no locally finite self-similar random trees with infinitely many ends: when applying the contraction operation repeatedly to such trees, the distance from the root to the branchpoints on the spine stochastically decreases, hence the degree of the root goes to infinity in law. Therefore, the degree of the root would have to be infinite in the first place.
\end{remark}

We now extend the above definitions to $\R$-trees\footnote{
We have been notified by an anonymous referee that one can maybe streamline some arguments by working in the setting of  $0$-hyperbolic spaces which includes both discrete trees and $\R$-trees. Since we are not familiar with these notions, we did not pursue this direction.}, and point the reader to Section~\ref{sec:Rtrees} for precisions concerning the definitions below.  We call $\Trinfall$ the space of complete, locally compact, rooted, measured $\R$-trees $\rT = (\rV,d,\root,\mu)$, with $\mu$ boundedly finite, modulo equivalence with respect to root- and measure-preserving isometries. This space is endowed with the GHP-topology and with its induced Borel $\sigma$-field. 
As above, we denote by $\Trinf\subset\Trinfall$ the subspace of trees with finitely many ends (see before Lemma \ref{lem:precompact_2} for the definition of an end in this setting),
and by $\Trfin$ and $\Trone$ the subspaces of $\Trinf$ for which the measure $\mu$ is finite or a probability measure, respectively. 

Every $\R$-tree $\rT$ defines a \emph{length measure} $\ell_\rT$ on its set of vertices, see  \eqref{eq:length}. In many cases of interest in probability theory, the length measure is not locally finite\footnote{With respect to the topology on $\rV$ induced by the metric $d$. One can define a finer topology generated by open segments with respect to which the length measure is always locally finite. Note that both topologies induce the same $\sigma$-algebra on the space of locally compact trees.}, for example in the case of Aldous's Brownian continuum random tree. However, in this article, we will be interested in those trees $\rT = (\rV,d,\root,\mu)\in\Trinf$ for which the measure $\mu$ dominates the length measure $\ell_\rT$, such that, in particular, the length measure is boundedly finite and 
$\mu$
has full support. We therefore define the spaces
\[
\Trlinf = \{\rT = (\rV,d,\root,\mu)\in\Trinf: \mu \ge \ell_\rT\},\quad \Trlfin = \Trfin \cap \Trlinf,\quad \Trlone = \Trone \cap \Trlinf,
\]
Note that a tree $\rT \in \Trlone$ has diameter at most 1 since $\mu$ dominates $\ell_\rT$ by definition. 
We show below that $\Trlinf$, $\Trlfin$ and $\Trlone$ are closed subspaces of $\Trinf$, $\Trfin$ and $\Trone$, respectively (Lemma~\ref{lem:Trlinf_closed}) and that the space $\Trlone$ is compact (Proposition~\ref{prop:topologies_2}).

We now define a (deterministic) rescaling operation $\Srop$ which will play the role of $\Sop$ for $\R$-trees. For a tree $\rT=(\rV,d,\root,\mu)\in\Trinf$ we denote by $\Sp(\rT)$ the subset of its vertices, called the \emph{spine}, which lie on an end. If $x\in\rV$, we denote by $\vec x$ the most recent ancestor of $x$ on the spine, i.e.\ the vertex in $\llbracket \root,x\rrbracket\cap \Sp(\rT)$ with maximal distance from the root. For two vertices $x,y\in \rV$, we then have
\[
 d(x,y) = d(x,\vec x) + d(\vec x,\vec y) + d(\vec y,y),\quad\tif \vec x \ne \vec y.
\]

\begin{definition}
 \label{def:Srop}
 If $\rT = (\rV, d, \root, \mu)\in \Trlinf$ and $p,q\in(0,1)$, then we define the tree $\rT' = \Srop(\rT) = (\rV', d', \root', \mu')$ by
\begin{itemize}
 \item $\rV' = \rV$ and $\root' = \root$,
 \item $d'(x,y) = pd(x,y) + (q-p)d(\vec x,\vec y)$ and
 \item $\mu' = p\mu + (q-p)\ell_\rT(\cdot \cap \Sp(\rT))$.
\end{itemize}
In words, we shrink distances \emph{off} the spine by a factor $p$ and \emph{on} the spine by a factor $q$ and scale the component $\mu-\ell_\rT$ of the measure $\mu$ by a factor $p$.
\end{definition}

As for $\Sop$, the map $\Srop: \Trlinf \to \Trlinf$ is not continuous when $p\ne q$, but is the limit as $R\to\infty$ of continuous maps $\Srop^R$ defined as follows: for  $\rT \in \Trlinf$ and $R \geq 0$ consider the subset $\Sp_R(\rT) \subset \rV$ of the vertices of $\rT$ that have a descendant at distance larger than $R$. 
Then define $\Srop^R$ analogously to $\Srop$, using $\Sp_R$ instead of $\Sp$\footnote{Formally, in the tree $\rT' = \Srop^R(\rT) = (\rV', d', \root', \mu')$, one has 
  $d'(x,y) = pd(x,y) + (q-p)d(\vec x_R,\vec y_R)$ and
  $\mu' = p\mu + (q-p)\ell_\rT(\cdot \cap \Sp_R(\rT))$ with 
  $\vec x_R$ the most recent ancestor of $x$ in $\Sp_R(\rT)$.}.
For every $R\ge0$, the map $\Srop^R$ is continuous, as can be shown by straightforward but fairly technical arguments,  see for example the proof of Lemma 2.6 (ii) in \cite{EPW06} for a similar situation. The map $\Srop = \lim_{R\to\infty} \Srop^R$ is then measurable and extends to a map $\Srop:\M{\Trlinf}\to\M{\Trlinf}$, allowing us to write $\Srop(\rT)$ for a \emph{random} tree $\rT\in\M{\Trlinf}$.

\begin{definition}
 \label{def:self_similar_continuous}
 Let $p,q\in(0,1)$. We say that a random rooted measured $\R$-tree $\rT\in\M{\Trlinf}$ is $(p,q)$-self-similar, if $\rT \stackrel{\text{law}}{=} \Srop(\rT)$.
\end{definition}
We now define a discretization operation on the space $\Trlinf$, which will allow to turn a self-similar $\R$-tree into a self-similar discrete tree. For this, we recall that as for discrete trees, every rooted $\R$-tree $\rT=(\rV,d,\root)$ induces a partial order $\preceq_\rT$ on its set of vertices $\rV$ by $x\preceq_\rT y$ if and only if $x\in\llbracket\root,y\rrbracket$, where $\llbracket\root,y\rrbracket$ is the range of the geodesic from $\root$ to $y$ (however, it is not true anymore that $\rT$ can be recovered from $\preceq_T$). Again, we write $x \prec_\rT y$ when $x \preceq_\rT y$ and $x \neq y$. We can now define the following discretization operation:

\begin{figure}[ht]
\label{fig:discret}
\begin{center}
 \includegraphics[width=10cm]{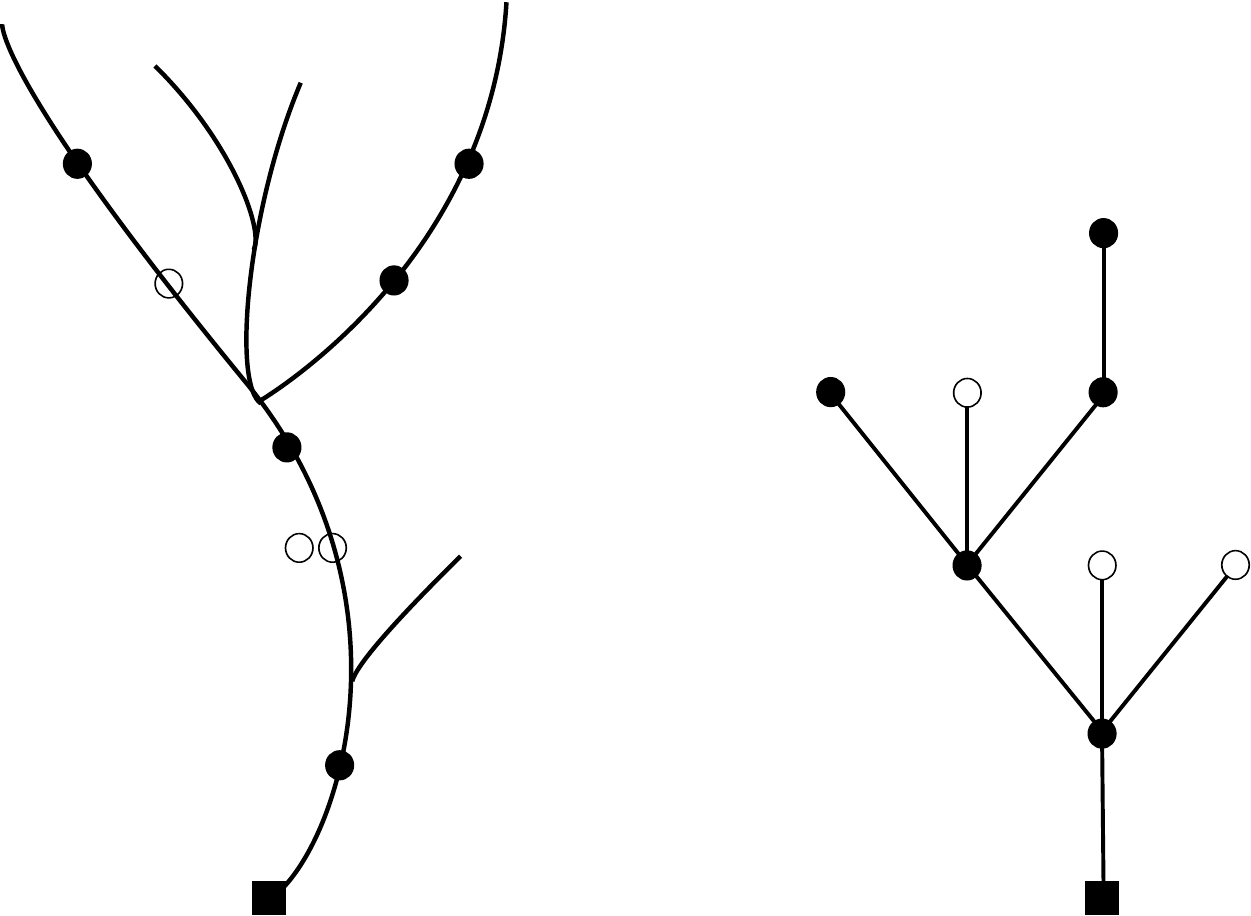}
\end{center}
\caption{An $\R$-tree $\rT$ (left) and a realization of its discretization $\Dop(\rT)$ (right). Black circles are vertices in $V_0$,
and white circles are vertices in $V_1$.
The square corresponds to the root. 
A vertex in $V_1$ is never an ancestor in $\Dop(\rT)$.
The two neighbouring white circles on the left indicate that two points have been sampled at the same spot; since vertices in $V_1$ are sampled according to the measure $\mu-\ell_\rT$, this is possible only if the measure $\mu$ has atoms.}
\end{figure}

\begin{definition}
\label{def:discret1}
Let $\rT=(\rV,d,\root)$ be a rooted $\R$-tree. Let $V_0$ be a subset of $\rV$ containing $\root$ and $V_1$ be a multiset of elements of $\rV$, formally, this can be defined as a counting measure on $\rV$. Suppose that $V_0$ and $V_1$ are boundedly finite in the sense that they contain a finite number of elements in each ball of finite radius. The \emph{discretized tree} $\Dop[\rT][V_0,V_1]$ is the rooted discrete tree $T$ with vertex set $V_0 \cup V_1$ and whose ancestral relation $\prec_{T}$ is defined as follows:
\[
\forall v,w\in V_0 \cup V_1: v \prec_{T} w \iff v \prec_\rT w \textbf{\tand} v \in V_0.
\]
\end{definition}
See Figure~\ref{fig:discret} for a graphical illustration.
\begin{example}
\label{ex:discret1}
If $\rT$ is the tree reduced to the root element $\root$, $V_1$ is the multi-set that contains $n$ times the root $\root$, and $V_0$ is the empty set, the discretized tree $\Dop[\rT][V_0,V_1]$ has $n$ edges adjacent to the root $\root$.
\end{example}
\begin{definition}
\label{def:discret2}
For a tree $\rT = (\rV,d,\root,\mu) \in\Trlinf$, define a random, rooted, discrete tree $\Dop[\rT]\in\M{\Tdinf}$ by $\Dop[\rT] = \Dop[\rT][V_0\cup\{\root\},V_1]$, where
\begin{itemize}
	\item $V_0$ is the set of atoms of a Poisson process on $\rV$ with intensity measure $\ell_\rT$, and
	\item $V_1$ is the multiset of atoms of a Poisson process on $\rV$ with intensity measure $\mu - \ell_\rT$.
\end{itemize}

\end{definition}
We will see below that the map $\Dop$ is actually continuous and thus can be extended to a map $\Dop:\M{\Trlinf}\to \M{\Tdinf}$ (see Section~\ref{sec:extension}). Standard properties of Poisson processes now yield the following commutation relation between $\Srop$, $\Sop$ and $\Dop$:
\begin{lemma}
 \label{lem:commutation}
We have the following equality of maps from $\M{\Trlinf}$ to $\M{\Tdinf}$:
$$\Dop \circ \Srop = \Sop \circ \Dop.$$
\end{lemma}

The (omitted) proof is based on the simple relation between Poisson processes on $\Srop(\rT)$ and $\rT$:
the Poisson process on $\Srop(\rT)$ is distributed as the thinning of the Poisson process on $\rT$, in which each atom is removed independently with a probability depending on its position: $1-q$ if the atom lies on $\Sp(\rT)$, and $p$ otherwise.

Lemma~\ref{lem:commutation} gives a way of constructing self-similar discrete trees from self-similar $\R$-trees. The following theorem, the main result of this article, says that all self-similar discrete trees arise this way.

\begin{theorem}
\label{th:self_similar}
Let $p,q\in(0,1)$. There exists a one-to-one correspondence between $(p,q)$-self-similar random rooted discrete trees $T \in \M{\Tdinf}$ and $(p,q)$-self-similar random rooted measured $\R$-trees $\rT\in\M{\Trlinf}$ given by 
\[
 T = \Dop[\rT].
\]
\end{theorem}

The methods used to prove Theorem~\ref{th:self_similar} will also allow to characterize \emph{compatible sequences} of random trees:

\begin{definition}
\label{def:compatible}
For a random rooted tree $T = (V,E,\root) \in \M{\Tdfin}$ with $\#V = n+1$ almost surely, define for $m\le n$, $\Cop{T}{m} = \Cop{T}{V'\cup\{\root\}}$, where $V'$ is a uniformly chosen subset of $V\backslash\{\root\}$ with $m$ (distinct) elements. A family of random rooted trees $(T_n)_{n\in \N}\in \M{\Tdfin}$ with $\#V(T_n)=n+1$ almost surely is called \emph{compatible}, if
for each $n \geq m \geq 1$,  we have $T_m \stackrel{\text{law}}{=} \Cop{T_n}{m}$.
\end{definition}

\begin{theorem}
\label{th:inverse_limit}
There exists a one-to-one correspondence between compatible families of random rooted trees $(T_n)_{n\in \N}$ and random rooted measured $\R$-trees $\rT\in \M{\Trlone}$, given by $T_n=\Dop[\rT][n]$ for $n\ge1$, where $\Dop[\rT][n]$ is the tree $\Dop[\rT]$ conditioned on having $n+1$ vertices.
\end{theorem}

The cornerstone in the proof of Theorems~\ref{th:self_similar} and~\ref{th:inverse_limit} will be the study of the continuity of the operator $\Dop$ and related topological issues. The following theorem summarizes the results thus obtained:
\begin{theorem}
\label{th:topologies}
The map $\Dop:\Trlinf\to\M{\Tdinf}$ as well as its extension $\Dop:\M{\Trlinf}\to \M{\Tdinf}$ are continuous, closed, injective maps. In other words, they are homeomorphisms onto their images and these are closed subsets of $\M{\Tdinf}$. 
\end{theorem}


\subsection*{Overview of the paper}


We start by proving Theorem~\ref{th:topologies} in Sections~\ref{sec:2} and \ref{sec:3}. In Section~\ref{sec:2} we consider the space $\Trlone$ only, the main result here is Proposition~\ref{prop:topologies_2} which shows that $\Trlone$ is compact and that the restriction of $\Dop$ to $\Trlone$ is continuous, closed and injective. The proof relies on the use of random distance matrices and certain exchangeable partial orders\footnote{We remark that exchangeable partial orders have been previously studied in generality by Janson \cite{Janson2011}, who provided a limiting representation 
based on the dense graph limits introduced by Lov\'{a}sz and Szegedy \cite{Lovasz2006}. 
It is however not clear to us how to make an efficient use of this representation for the questions studied here.}, together with a somewhat intricate analysis of the continuity of a certain class of polynomial test functions\footnote{A previous version of this article (published on the arXiv) contained a different, but incomplete proof.}. 

In Section~\ref{sec:3}, we prove Theorem~\ref{th:topologies} in full generality, i.e.\ on the space $\Trlinf$ of infinite trees. This is probably the most technical section; we make use several different ways of truncating the trees and some technical arguments to bound the number of ends in the trees. Precompactness arguments play an important role.

In Section~\ref{sec:4}, Theorems~\ref{th:self_similar} and \ref{th:inverse_limit} are proven. At the heart of the proofs is the following idea: First, we turn a discrete tree into an $\R$-tree by assigning length $1$ to each edge. We then show that rescaling and then discretizing that $\R$-tree yields with high probability the same result as contracting the original tree, one an arbitrarily large ball (Lemma~\ref{lem:shrink_coupling}). Together Theorem~\ref{th:topologies}, this allows to construct the real trees $\rT$ in Theorems~\ref{th:self_similar} and \ref{th:inverse_limit} as scaling limits of the discrete tree $T$, respectively, the sequence of discrete trees $T_n$.

In Section~\ref{sec:5}, motivated by the correspondence between self-similar discrete trees and self-similar $\R$-trees established in Theorem~\ref{th:self_similar}, we study examples of self-similar $\R$-trees. We give there an overview of the generality of examples that can be constructed.
Like in the case of self-similar real-valued processes, see~\cite{O'Brien1985}, it seems out of reach to completely characterize this family. We therefore consider in Section~\ref{sec:6} a specific class of $(p,q)$-self-similar trees, namely those that are invariant with respect to translation along the spine (we suppose here for simplicity that the spine consists of a single infinite ray). In particular, in the case of self-similar trees consisting of a single spine to which iid subtrees are attached, we relate the construction of the corresponding $\R$-trees to the quasi-stationary distributions of linear-fractional subcritical Galton--Watson processes, see  Proposition~\ref{prop:iid} and  Remark~\ref{rem:qsd}.

In the short Section~\ref{sec:7}, we describe another attempt to prove Theorem~\ref{th:inverse_limit} using exchangeability, which we initially pursued but dropped because of its drawbacks.

Finally, an appendix recalls some notions on the space of discrete trees and $\R$-trees that we consider in this work.

\subsection*{Acknowledgments} 
We are grateful to Itai Benjamini for asking us a question which motivated this study. 
We are also deeply indebted to Ohad Feldheim; example 1 in Section~\ref{sec:5} was found 
following discussions with him before the general statement of Theorem~\ref{th:self_similar} was clear to us. Further, we thank Tom Meyerovitch, Gr\'egory Miermont and Ron Peled for useful discussions. Finally, an anonymous referee gave several useful comments and informed us about the article \cite{ALW15}.

\section{The map \texorpdfstring{$\Dop$}{D}: finite trees}
\label{sec:2}
In this section, we study the action of $\Dop$ on the space $\Trlone$, i.e.\ those measured trees in $\Trlinf$ whose measure is a probability measure. The results are summarized in the following proposition:
\begin{proposition}
 \label{prop:topologies_2}
The restriction of the map $\Dop$ to $\Trlone$ is a homeomorphism onto its image. Furthermore, the spaces $\Trlone$ and $\Dop(\Trlone)$ are compact.
\end{proposition}

\def\DM{\mathtt{DM}}
\def\DMT{\widetilde{\DM}}

\begin{figure}[h]
 \def\svgwidth{10cm}
 \centering
 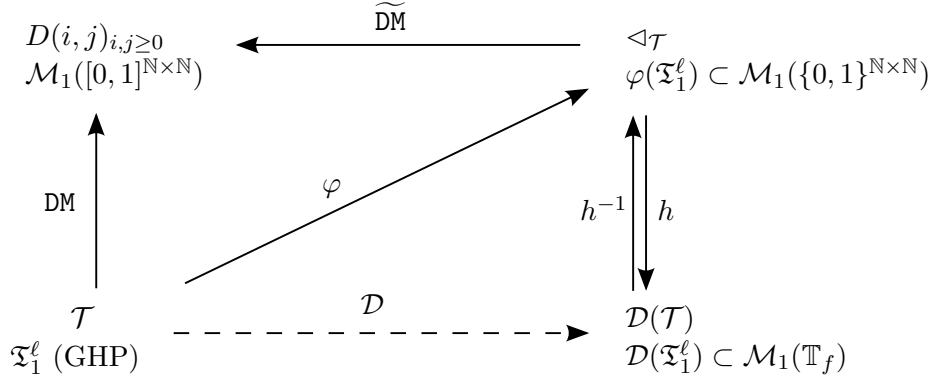
 \caption{The different spaces used to prove Proposition~\ref{prop:topologies_2} and the maps between them.}
 \label{fig:diagram}
\end{figure}

For this we rely on two different representations of random rooted measured trees, one in terms of its distance matrix distribution defined in Section~\ref{sec:Rtrees}, the other in terms of an exchangeable partial order on $\N$ (the relationships between these representations is summarized in Figure~\ref{fig:diagram}). Set $\N^*=\{1,2,\ldots\}$. We recall that given a tree $\rT = (\rV,d,\root,\mu)\in\Trone$, its \emph{distance matrix distribution} $\DM(\rT)$ is defined as the push-forward of the probability measure $\delta_\root\otimes \mu^{\otimes \N^*}$ by the map 
$$(x_i)_{i \in \N} \mapsto d(x_i,x_j)_{(i, j) \in \N \times \N}.$$
Since the distance between any two points of the tree $\rT\in\Trlone$ is less than or equal to 1, its distance matrix distribution
 $\DM(\rT)$ is a probability measure on $[0,1]^{\N\times\N}$, endowed with the product topology.
 
The usefulness of the distance matrix distribution comes from the fact that it convergence determining as recalled in the appendix. Here, we will use the following fact, which follows from the second part of Lemma~\ref{lem:precompact} together with Lemma~\ref{lem:GP_GHP}:
\begin{lemma}
 \label{lem:super_agreable}
 A sequence $\rT_n\in\Trlone$ converges to a limit $\rT\in\Trinfall$ if and only if its distance matrix distributions $\DM(\rT_n)$ converge. In this case, $\rT\in\Trone$ and $\DM(\rT) = \lim_{n\to\infty} \DM(\rT_n)$.
\end{lemma}

Denote by $\PSpace$ the space of random partial orders $\EPO$ on $\N$ which are invariant under finite permutations of $\N^*$ and such that $0\EPO n$ for every $n\in\N^*$. 
Through the map $\EPO \, \mapsto (\ind_{i \EPO j})_{(i, j) \in \N \times \N}$,
we can identify $\PSpace$ with a closed subspace of the space of probability measures on $\{0,1\}^{\N\times\N}$, so that $\PSpace$ is compact. Every tree $\rT\in\Trlone$ then defines an element $\varphi(\rT)$
in $\PSpace$ as follows: Fix a representative of the tree $\rT$, also denoted by $\rT$.
Let  $(X_1,S_1),(X_2,S_2),\ldots$ be an iid sequence of random variables in $\rV\times \{0,1\}$ with law $\ell_\rT\otimes \delta_0 + (\mu-\ell_\rT)\otimes \delta_1$. This means that $X_1,X_2,\ldots$ are iid according to $\mu$ and  $S_i = 0$ if $X_i$ was drawn according to $\ell_\rT$ and $S_i = 1$ otherwise. Set $(X_0,S_0) = (\root,0)$. We then define the transitive relation $\EPO[\rT]$ on $\N$ by
\begin{equation}
i \EPO[\rT] j \iff X_i \prec_{\rT} X_j\tand S_i = 0,
\label{eq:EPO}
\end{equation}
and denote its law by $\varphi(\rT)$. The relation $\EPO[\rT]$ is not reflexive, therefore is not a partial order.

Note that for a tree $\rT\in\Trlone$, the 
binary relation
$\EPO[\rT]$ is in direct relation to the discretization $\Dop(\rT)$. Indeed, if one denotes by $\EPO[\rT]^N$ the restriction of $\EPO[\rT]$ to the elements $\{0,\ldots,N\}$, with $N\sim\Poi(1)$, then the random rooted tree on $N+1$ vertices with ancestral relation $\EPO[\rT]^N$ has the same distribution as $\Dop(\rT)$. Conversely, given $\Dop(\rT)$, one can define the restrictions $\EPO[\rT]^n$, $n=1,2,\ldots$ by conditioning the tree $\Dop(\rT)$ on having $n+1$ vertices, assigning random labels $\{1,\ldots,n\}$ to the non-root vertices and then defining $\EPO[\rT]^n$ as the ancestral relation of this tree. This uniquely defines $\EPO[\rT]$.

We thus have constructed a bijection $h:\varphi(\Trlone)\to \Dop(\Trlone)$ such that $\Dop = h\circ \varphi$. Moreover, this bijection is a homeomorphism because for every $n$, the map assigning the law of $\Dop(\rT)$ conditioned on having $n+1$ vertices to the law of $\EPO[\rT]^n$ is continuous (as well as its inverse) by finiteness of the corresponding spaces. We have thus proven the following lemma:

\begin{lemma}
 \label{lem:homeomorphism-0}
There exists a homeomorphism $h:\varphi(\Trlone)\to\Dop(\Trlone)$ such that $\Dop = h\circ\varphi$.
\end{lemma}


As a consequence of the previous result, in order to show Proposition~\ref{prop:topologies_2}, it will be enough to prove the following lemma.

\begin{lemma}
 \label{lem:homeomorphism}
$\varphi$ is a homeomorphism between $\Trlone$ and its image $\varphi(\Trlone)$. Moreover, both spaces are compact.
\end{lemma}

In order to prove Lemma~\ref{lem:homeomorphism}, we will make a detour by the space of random distance matrices.
To wit, to 
every element $\EPO$ from $\PSpace$, we associate a random distance matrix $\widetilde D(i,j)_{i,j\ge0}$ by 
\begin{equation}
\label{eq:Dtilde_bis}
 \widetilde D(i,j) = \lim_{n\to\infty} \frac 1 n \sum_{k=1,k\not\in\{i,j\}}^n \ind_{k \EPO i,\, k \not\EPO j\text{ or }k \EPO j,\, k \not\EPO i}.
\end{equation}
The existence of the limit is provided by de Finetti's theorem applied to the exchangeable sequence $(\ind_{k \EPO i,\, k \not\EPO j\text{ or }k \EPO j,\, k \not\EPO i})_{k\ne i,j}$.
One easily verifies that $\widetilde D$ satisfies the triangle inequality almost surely, whence we call it a random distance matrix. We then denote the map associating the law of the random distance matrix $\widetilde D$ to the (law of)
$\EPO$ by $\DMT:\PSpace \to \M{[0,1]^{\N\times\N}}$. We have the following lemma.

\begin{lemma}
 \label{lem:DMtilde}
 $\DM = \DMT \circ \varphi$ on $\Trlone$.
\end{lemma}
\begin{proof}
 Fix a representative of a tree $\rT = (\rV,d,\root,\mu)\in\Trlone$. Let $((X_i,S_i))_{i\ge0}$ be as above and define $D(i,j) = d(X_i,X_j)$, such that $D$ follows the law $\DM(\rT)$. By the law of large numbers, this gives,
\begin{align*}
 D(i,j) = \ell_\rT(\llbracket X_i,X_j\rrbracket) &= \lim_{n\to\infty}\frac 1 n \sum_{k=1,k\not\in\{i,j\}}^n \ind_{X_k\in\llbracket X_i,X_j\rrbracket,\,S_k = 0}\\
 &= \lim_{n\to\infty}\frac 1 n \sum_{k=1,k\not\in\{i,j\}}^n \ind_{k \EPO[\rT] i,\, k \not\EPO[\rT] j\text{ or }k \EPO[\rT] j,\, k \not\EPO[\rT] i},
\end{align*}
where the last equality follows from the definition of $\EPO[\rT]$. Equation~\eqref{eq:Dtilde_bis} then shows that $D$ follows the law $\DMT(\EPO[\rT])$, which proves the lemma.
\end{proof}

\begin{lemma}
\label{lem:Trlone_compact}
 The space $\Trlone$ is compact.
\end{lemma}
\begin{proof} 
By the first part of Lemma~\ref{lem:precompact}, the space $\Trlone$ is precompact in $\Trinfall$. It thus suffices to show that $\Trlone$ is closed. 

We first introduce some more notation. Fix a representative of a tree $\rT = (\rV,d,\root,\mu)\in\Trone$. Let $(X_i)_{i\ge0}$ be as above and define $D(i,j) = d(X_i,X_j)$, such that $D$ follows the law $\DM(\rT)$. Now set for $m\in\N^*$,
 \begin{equation}
  \label{eq:Mm}
  M^{(m)}(1,2) = \frac 1 m \sum_{k=3}^m \ind_{D(1,k) + D(k,2) = D(1,2)}.
 \end{equation}
By the law of large numbers, $M(1,2) = \lim_{m\to\infty} M^{(m)}(1,2)$ exists almost surely and equals $\mu(\llbracket X_1,X_2\rrbracket)$. Moreover, conditioned on $M(1,2)$, the random variables $Y_k = \ind_{D(i,k) + D(k,j) = D(i,j)}$, $k=3,4,\ldots$ are iid Bernoulli with parameter $M(1,2)$. By the conditional Chebychev inequality, we therefore have for every $\ep>0$,
\begin{equation}
 \label{eq:conv_uniform}
 \P(|M^{(m)}(1,2) - M(1,2)| > \ep|) \le  \frac{1}{m \ep^2} \E[\operatorname{Var}(Y_3\,|\,M(1,2))] \le \frac{1}{4m\ep^2}.
\end{equation}
In other words, the convergence of $M^{(m)}(1,2)$ to $M(1,2)$ is uniform in $\rT$.

Now, suppose  there exist $\rT_1,\rT_2,\ldots \in \Trlone$ such that $\rT_n$ converges to $\rT\in\Trinfall$ as $n\to\infty$. By Lemma~\ref{lem:super_agreable}, $\rT\in\Trone$ and  $\DM(\rT_n) \to \DM(\rT)$ as $n\to\infty$. Define $D$, $M(1,2)$ and $M^{(m)}(1,2)$ as above and define $D_n$, $M_n(1,2)$ and $M_n^{(m)}(1,2)$ analogously for every $n\in\N$. By Skorokhod's representation theorem we can and will assume that the distance matrices $D_n$ converge (pointwise) almost surely to $D$. We then have almost surely, for every $m\in\N$,
\[
 M^{(m)}(1,2) \ge \lim_{n\to\infty} M_n^{(m)}(1,2),
\]
since the RHS in \eqref{eq:Mm} is an upper semi-continuous function in $D$ for every $m$. The uniform convergence of $M^{(m)}(1,2)$ to $M(1,2)$ proven in \eqref{eq:conv_uniform} then shows that $M(1,2) \ge \lim_{n\to\infty} M_n(1,2).$ Moreover, since $\rT_n\in\Trlone$ for every $n$, we have $M_n(1,2) \ge D_n(1,2)$ almost surely for every $n$, so that almost surely $$M(1,2) \ge \lim_{n\to\infty} D_n(1,2) = D(1,2).$$
But since $M(1,2) = \mu(\llbracket X_1,X_2\rrbracket)$ and $D(1,2) = \ell_\rT(\llbracket X_1,X_2\rrbracket)$ with $X_1,X_2$ iid according to $\mu$, this implies that $\mu \ge \ell_{\rT}$ on its support. By Lemma~\ref{lem:precompact}, the measure $\mu$ has full support, whence $\mu\ge \ell_{\rT}$. This shows that $\rT\in\Trlone$ and hence $\Trlone$ is a closed subspace of the compact space $\Trone$.
\end{proof}

\begin{lemma}
\label{lem:varphi_continuous}
 The map $\varphi$ is continuous on $\Trlone$.
\end{lemma}
\begin{proof}
We will show more in fact: we show that the map which to a tree $\rT\in\Trlone$ assigns the law of $(D_\rT,\EPO[\rT])$ is continuous, where $D_\rT$ is the distance matrix of the tree $\rT$. For this, we will consider test functions of the form
\begin{equation}
\label{eq:D_EPO_test_fn}
 f(D,\EPO) = C \prod_{i,j=0}^n D(i,j)^{\beta_{ij}} \prod_{l=1}^L \ind_{a_l \EPO b_l},
\end{equation}
where $C\in\R$, $n\in\N$, $\beta_{ij}\in\N$, $L\ge 0$ and $a_l,b_l\in\{1,\ldots,n\}$. Note that $D_\rT(i,j)\le 1$ almost surely for all $\rT\in\Trlone$ and $i,j\in\N$, whence we can formally see the couple $(D,\EPO)$ as a random variable taking values in the compact space $[0,1]^{\N\times\N}\times\{0,1\}^{\N^*\times\N^*}$. The vector space spanned by functions of the form \eqref{eq:D_EPO_test_fn} then forms an algebra of continuous functions on this space which separates points. By the Stone-Weierstrass theorem it is therefore enough to show that for a function $f$ as in \eqref{eq:D_EPO_test_fn}, the quantity $\E[f(D_\rT,\EPO[\rT])]$ is continuous in $\rT$. We will show this by induction on $L$. If $L=0$, the assertion follows from the fact that the law of $D_\rT$ is continuous in $\rT$ by Lemma~\ref{lem:GP_GHP} and Fact~\ref{fact:Gweak}. If $L>0$, define the sets $A = \{a_l:l=1,\ldots,L\}$ and $B = \{b_l:l=1,\ldots L\}$. We will distinguish two cases:

\underline{Case $A\subset B$}: In this case, there must exist a cycle\footnote{To see this, start with $l_1 = 1$ and let $l_2$ be such that $a_{l_1} = b_{l_2}$. Then let $l_3$ be such that $a_{l_2} = b_{l_3}$ and so on. Since the $l_i$'s take values in the finite set $\{1,\ldots,L\}$, a cycle has to appear at some point.} $l_1,\ldots,l_k,l_{k+1} = l_1$, such that $a_{l_i} = b_{l_{i+1}}$ for all $i=1,\ldots,k$. In particular, $f(D,\EPO)\ne 0$ implies that 
\[
b_{l_1} = b_{l_{k+1}} = a_{l_k} \EPO b_{l_k} = a_{l_{k-1}} \EPO \cdots \EPO b_{l_1},
\]
whence, by transitivity, $b_{l_1} \EPO b_{l_1}$. But by \eqref{eq:EPO}, we have $k \not\EPO[\rT] k$ for all $k\in\N$, whence $f(D_\rT,\EPO[\rT]) \equiv 0$. In particular, $\E[f(D_\rT,\EPO[\rT])]$ is continuous in $\rT$.

\underline{Case $A\not\subset B$}: In this case, let $\alpha\in A\backslash B$.
Define the sets 
\[
\Lambda = \{l\in\{1,\ldots,L\}: a_l = \alpha\},\quad \overline{\Lambda} = \{1,\ldots,L\}\backslash \Lambda.
\]
Note that $\Lambda \ne \emptyset$, such that $\#\overline{\Lambda} < L$. We will show that we can express $\E[f(D_\rT,\EPO[\rT])]$ as a sum of expressions of the same type, with functions of the form \eqref{eq:D_EPO_test_fn} containing only the indicator functions corresponding to $l\in\overline{\Lambda}$. This will allow us to complete the induction step.

Recall the construction of $D_\rT$ and $\EPO[\rT]$ from the sequence $(X_0,S_0),(X_1,S_1),\ldots$ defined above. Since $\Lambda\ne\emptyset$, we have by definition,
\begin{equation}
\label{eq:ste_genevieve}
 \prod_{l\in\Lambda} \ind_{a_l\EPO[\rT] b_l} = \ind_{S_\alpha = 0}\prod_{l\in\Lambda}\ind_{X_\alpha \prec_\rT X_{b_l}}.
\end{equation}
Now define $B_\Lambda = \{b_l:l\in\Lambda\}$. We can assume that $\alpha \not\in B_\Lambda$, otherwise the function $\ind_{\alpha\EPO\alpha}$ would be a factor of $f(D,\EPO)$ and therefore $f(D_\rT,\EPO[\rT]) \equiv 0$. Let $X'$ be the most recent common ancestor of the vertices $X_b$, $b\in B_\Lambda$, which is $\sigma(X_b;\,b\in B_\Lambda)$-measurable. By the definition of the sequence $(X_i,S_i)_{i\ge0}$, conditioned on $(X_k,S_k)_{k\ne \alpha}$, the event $\{S_\alpha = 0, X_\alpha \prec_\rT X'\}$ has probability $d_\rT(\root,X')$ and conditioned on this event, $X_\alpha$ is uniformly distributed on $\llbracket \root,X'\rrbracket$. Furthermore, if $X'\wedge X_i$ denotes the most recent common ancestor of $X'$ and $X_i$, then if $X_\alpha\in\llbracket \root,X'\rrbracket$, we have for every $i\ne \alpha$, 
\begin{align*}
D_\rT(\alpha,i) = d_\rT(X_\alpha,X_i) 
&= d_\rT(X_\alpha,X'\wedge X_i) + d_\rT(X'\wedge X_i,X_i)\\
&= |d_\rT(\root,X_\alpha) - d_\rT(\root,X'\wedge X_i)| + d_\rT(\root,X_i) - d_\rT(\root,X'\wedge X_i).
\end{align*}

We now use this to calculate the expectation of the factors of $f(D_\rT,\EPO_\rT)$ involving $\alpha$, conditioned on the remaining ones. Since $D_\rT$ is symmetric, we can assume that $\beta_{i\alpha} = 0$ for all $i\ne \alpha$, and furthermore, $\beta_{\alpha\alpha} = 0$, since $D_\rT(\alpha,\alpha) = 0$ almost surely. From the above, we then have
\begin{multline}
\label{eq:prod_U}
 \E\Big[\prod_{i\in\{0,\ldots,n\}\backslash\alpha} D(\alpha,i)^{\beta_{\alpha i}} \ind_{(S_\alpha = 0,\,X_\alpha \EPO[\rT] X')}\,\Big|\, (X_k,S_k)_{k\ne \alpha}\Big] \\
 = \E_U\Big[\prod_{i\in\{0,\ldots,n\}\backslash\alpha} \big(|U-x_i|+y_i\big)^{\beta_{\alpha i}}\ind_{(U < z)}\Big],
\end{multline}
where $U\sim \operatorname{Unif}(0,1)$, $\E_U$ denotes expectation w.r.t. $U$ and $x_i,y_i,z$ are
\[
x_i = d_\rT(\root,X'\wedge X_i),\quad y_i = d_\rT(\root,X_i) - d_\rT(\root,X'\wedge X_i),\quad z = d_\rT(\root,X').
\]
Each factor in the product on the RHS of \eqref{eq:prod_U} can be written as
\[
 \big(|U-x_i|+y_i\big)^{\beta_{\alpha i}}\ind_{(U < z)} = \big(U-x_i - 2(U-x_i)\ind_{(U<x_i)} - y_i\big)^{\beta_{\alpha i}}\ind_{(U < z)}.
\]
Expanding the powers and the product then gives that
\[
 \prod_{i\in\{0,\ldots,n\}\backslash\alpha} \big(|U-x_i|+y_i\big)^{\beta_{\alpha i}}\ind_{(U < z)} = \sum_{\xi\in\{x_i,i\ne\alpha\}\cup \{z\}} P_\xi(U,(x_i)_{i\ne \alpha},(y_i)_{i\ne\alpha})\ind_{(U<\xi)},
\]
where each $P_\xi$ is a polynomial in its arguments. Taking expectations and using the fact that $\E_U[U^{k-1}\ind_{(U<\xi)}] = \xi^k/k$ for every $k\ge1$ and $\xi\in[0,1]$, we obtain that the RHS of \eqref{eq:prod_U} is a polynomial in $(x_i)_{i\ne \alpha}$, $(y_i)_{i\ne\alpha}$ and $z$. 

Now note that for every $i\ne \alpha$, $d_\rT(\root,X')$ and $d_\rT(\root,X'\wedge X_i)$ are linear combinations\footnote{For two vertices $x,y$, if $z$ denotes their most recent common ancestor, $d_\rT(\root,z) = \frac 1 2(d_\rT(\root,x) + d_\rT(\root,y) - d_\rT(x,y))$. The statement follows easily from this by induction.} of $d_\rT(X_i,X_j) = D_\rT(i,j)$, $i,j\in\{0,\ldots,n\}\backslash\{\alpha\}$. Together with the above, this implies that the RHS in \eqref{eq:prod_U} is in fact a polynomial in $D_\rT(i,j)$, $i,j\in\{0,\ldots,n\}\backslash\{\alpha\}$.
Plugging this into \eqref{eq:D_EPO_test_fn} and using \eqref{eq:ste_genevieve} yields that
\(
 \E[f(D_\rT,\EPO[\rT])]
\)
can be written as the sum of expressions of the same type, with functions of the form \eqref{eq:D_EPO_test_fn} with $L$ replaced by $\#\overline{\Lambda} < L$. This finishes the induction step and therefore the proof of the lemma.
\end{proof}

\begin{proof}[Proof of Lemma~\ref{lem:homeomorphism}]
Since $\Trlone$ is compact by Lemma~\ref{lem:Trlone_compact} and $\varphi(\Trlone)$ is Hausdorff, it is enough to show that $\varphi$ is bijective and continuous \cite[Theorem XI.2.1]{Dugundji}. Continuity follows from Lemma~\ref{lem:varphi_continuous}, and bijectivity follows from Lemma~\ref{lem:DMtilde} and the fact that $\DM$ is injective by Fact~\ref{fact:Gromov_weak_uniqueness}.
\end{proof}

\begin{proof}[Proof of Proposition~\ref{prop:topologies_2}]
Follows immediately from Lemmas \ref{lem:homeomorphism-0} and \ref{lem:homeomorphism}.
\end{proof}

\section{The map \texorpdfstring{$\Dop$}{D}: infinite trees (proof of Theorem~\ref{th:topologies})}
\label{sec:3}

In this section, we prove Theorem~\ref{th:topologies}. The proof uses a series of lemmas. The first one concerns the continuity of the map $\Dop$:
\begin{lemma}
 \label{lem:D_continuous}
 The map $\Dop:\Trlinf\to\M{\Tdinf}$ is continuous.
\end{lemma}

In view of Lemma~\ref{lem:D_continuous} and Section~\ref{sec:extension}, we can extend $\Dop$ to a continuous map $\Dop:\M{\Trlinf}\to\M{\Tdinf}$. The next lemma shows that this extension is injective.

\begin{lemma}
 \label{lem:D_injective}
 The map $\Dop:\Trlinf \to \M{\Tdinf}$ and its extension $\Dop:\M{\Trlinf}\to\M{\Tdinf}$ are injective.
\end{lemma}

The next lemma will be used to prove that the image $\Dop$ is closed in $\M{\Tdinf}$:
\begin{lemma}
 \label{lem:tightness}
If $\rT_1,\rT_2,\ldots\in\Trlinf$ (resp., $\M{\Trlinf}$), such that $\Dop(\rT_n)$ converges in law to a random tree $T$ supported on $\Tdinf$, then there exists $\rT\in\Trlinf$ (resp., $\M{\Trlinf}$), such that $\Dop(\rT) = T$ and $\rT_n\to\rT$.
\end{lemma}

The proof of Theorem~\ref{th:topologies} now directly follows from the previous results:
\begin{proof}[Proof of Theorem~\ref{th:topologies}]
Continuity, injectivity and closedness of the map  $\Dop$ and its extension are exactly Lemmas~\ref{lem:D_continuous}, \ref{lem:D_injective} and \ref{lem:tightness}, respectively.
\end{proof}

One ingredient for Lemma~\ref{lem:tightness} is the following lemma which is of independent interest:
\begin{lemma}
\label{lem:Trlinf_closed}
$\Trlinf$, $\Trlfin$ and $\Trlone$ are closed subspaces of $\Trinf$, $\Trfin$ and $\Trone$, respectively.
\end{lemma}

The proofs of the four lemmas follow.

\begin{proof}[Proof of Lemma~\ref{lem:D_continuous}]
The continuity on the space $\Trlone$ follows from Proposition~\ref{prop:topologies_2}. 
There is the map: 
$$ \Trlfin \to \Trlone \times \R^+, \quad \rT \mapsto \left(\Srop[\mu(\rV)^{-1}](\rT),\mu(\rT) \right)$$
that is continuous, therefore the continuity  extends to the space $\Trlfin$.
We now turn to $\Trlinf$. We consider $\rT_n \in \Trlinf$ with limit $\rT$, and
we need to prove that 
\begin{equation}
\label{eq:stilton}
\Dop(\rT_n)^{\le m}\to \Dop(\rT)^{\le m}, \text{for every } m \in \N.
\end{equation}
The convergence of $\rT_n$ implies by \ref{eq:inmeasure} for 
each $R\ge0$ the existence of a sequence $R_n \ge R$ converging to $R$ and such that $\rT_n^{\le R_n} \to \rT^{\le R}$ . The convergence
\begin{equation}
\label{eq:stichelton}\Dop(\rT_n^{\le R_n}) \to \Dop(\rT^{\le R})
 \end{equation}
follows since the truncated trees $\rT_n^{\le R_n}, \rT^{\le R}$ belong to $\Trlfin$, on which the map $\Dop$ is continuous.
Fix $m \in \N$ and $\ep>0$. Throughout the rest of the proof, we write $\rT_\infty = \rT$ and let $n$ take values in $\N\cup\{\infty\}$. The convergence in \eqref{eq:stilton} now follows from \eqref{eq:stichelton}, $R_n \ge R$ and the following claim: there exist $R,n_0<\infty$, such that 
\begin{equation}
\label{eq:ched}
\P(\Dop(\rT_n)^{\le m} \subset \Dop(\rT_n^{\le R})) > 1-\ep\quad\text{for $n\ge n_0$.}
\end{equation}
To establish the claim, consider the paths in $\rT_n$ of length $R$ starting from the root.
The number of these paths that have a distinct restriction to a distance $r$ from the root, $r \leq R$, is $N_{r,R}(\rT_n)$ defined before Lemma~\ref{lem:precompact_2}. 
By that lemma, for every $r\ge 0$, there exist $R,n_0 < \infty$ such that $N_{r,R}(\rT_n)\le N$ for $n\ge n_0$, where $N$ only depends on the family $(\rT_n)_{n\in\N\cup\{\infty\}}$ and not on $m,r,R,n_0$. 
Now, the number of points in $V_0$ in a path of length $r$ (in the discretization of $\rT_n$) 
 is a $\operatorname{Poi}(r)$-distributed random variable $X$, and we  can take $r$ large enough, so that $\P(X < m) < \ep/N$. A union bound then shows that, with probability at least $1-\ep$, the restriction to distance $r$ of any path in $\rT_n$ of length $R$ starting from the root has at least $m$ points in $V_0$, for $n \ge n_0$. This proves
\eqref{eq:ched}, and \eqref{eq:stilton} follows.
\end{proof}

\begin{proof}[Proof of Lemma~\ref{lem:D_injective}]
Since $\Dop(\rT) = \Dop(\delta_\rT)$ for every $\rT\in\Trlinf$, it suffices to show the injectivity of the extension only.
The map $\DM$ is injective on the space $\M{\Trlone}$ by Fact~\ref{fact:Gromov_weak_uniqueness}. By Lemma~\ref{lem:commutation}, the extension of the map $\varphi$ to $\M{\Trlone}$ then is injective as well. Furthermore, the bijective correspondence between $\varphi(\rT)$ and $\Dop(\rT)$ for $\rT\in\Trlone$ described before Lemma~\ref{lem:homeomorphism} readily extends to random trees. This immediately shows injectivity of $\Dop$ on $\M{\Trlone}$.

The injectivity of $\Dop$ on $\M{\Trlfin}$ follows using the rescaling argument in the proof of Lemma~\ref{lem:D_continuous}.

To prove the injectivity of $\Dop$ on $\M{\Trlinf}$, we introduce two pruning operations, on $\R$-trees and on discrete trees, that commute with $\Dop$.
Let $\rT$ be a random tree in $\Trlinf$. For $\lambda > 0$, let $\rT^{\lambda}$ be the random tree obtained by cutting the tree $\rT$ at the atoms of a Poisson process with intensity $\lambda \ell_\rT|_{\Sp(\rT)}$, and keeping the component containing the root.
Also, let $T$ be a random tree  in $\Tdinf$. For $\lambda >0$, let $T^{\lambda}$ be the random tree obtained from $T$ by first marking every vertex on the spine independently with probability $\lambda/(1+\lambda)$, then removing all the marked vertices (together with the adjacent edges) and keeping the component containing the root.
There is the equality in distribution\footnote{This relies on the fact that in the discretization operation the spine vertices are necessarily in $V_0$ and thus have been sampled according to the length measure $\ell_\rT$.}$$\Dop(\rT^{\lambda})=\Dop(\rT)^{\lambda}.$$

We are now ready to complete the proof. Consider $\rT_1$ and $\rT_2$ two random trees in $\Trlinf$ that satisfy $\Dop(\rT_1) = \Dop(\rT_2)$. We then have:
\begin{equation}
\label{eq:cheltenham}
 \Dop(\rT^{\lambda}_1) = \Dop(\rT_1)^{\lambda} = \Dop(\rT_2)^{\lambda} = \Dop(\rT_2^{\lambda}).
\end{equation}
Also, $\rT_1^{\lambda}$ and $\rT_2^{\lambda}$ are elements of $\M{\Trlfin}$. This may be justified as follows (for $\rT_1$ say): the tree $\Dop(\rT_1)$ has an  almost surely finite number of infinite paths starting from the root and every one of them is truncated in $\Dop(\rT_1)^{\lambda}$ at a finite distance from the root, therefore the tree $\Dop(\rT_1)^\lambda$ has a finite diameter and, being locally finite, it has finitely many vertices. Since the number of non-root vertices in $\Dop(\rT_1)^\lambda$ is a Poi($\mu(\rT^\lambda_1)$)-distributed random variable, this entails that $\mu(\rT_1^\lambda) < \infty$ almost surely, and since $\mu\ge \ell_{\rT_1}$, that $\rT_1^\lambda\in \M{\Trlfin}$ almost surely. From \eqref{eq:cheltenham} and the injectivity of $\Dop$ on $\M{\Trlfin}$, we deduce that $\rT_1^{\lambda} = \rT_2^{\lambda}$, and the identity $\rT_1 = \rT_2$ follows taking the $\lambda \to 0$ limit.
\end{proof}

\begin{proof}[Proof of Lemma~\ref{lem:Trlinf_closed}]
The space $\Trlone$ is compact according to Lemma~\ref{lem:Trlone_compact}. 
In particular, it is closed in $\Trone$. 
It remains to prove that $\Trlinf$, $\Trlfin$ are closed subspaces of $\Trinf$, $\Trfin$.
For $\Trlfin$, this follows by rescaling 
as in the proof of Lemma \ref{lem:D_continuous}.
For $\Trlinf$, we use truncation:
If $\rT_n \in\Trlinf$ converge to  $\rT \in \Trinf$, then, by \ref{eq:inmeasure}, for 
every $R>0$, there exists a sequence $R_n \ge R$ converging to $R$  
such that $\rT_n^{\le R_n} \to \rT^{\le R}$, with $\rT_n^{\le R_n} \in \Trlfin$ and $\rT^{\le R} \in \Trfin$.
Since $\Trlfin$ is closed in $\Trfin$, we have that $\rT^{\le R} \in \Trlfin$ for every $R$, whence $\mu$ dominates $\ell_{\rT}$ on every ball of radius $R$ around the root in $\rT$. Hence, $\mu\ge \ell_{\rT}$ everywhere and thus $\rT\in\Trlinf$.
\end{proof}

\begin{proof}[Proof of Lemma~\ref{lem:tightness}]
We only prove the Lemma for $\M{\Trlinf}$, the proof for $\Trlinf$ is similar.
Assume that $(\rT_n = (\rV_n,d_n,\root_n,\mu_n))_{n\ge0}$ is a sequence of random trees taking values in $\M{\Trlinf}$, such that $\Dop(\rT_n)$ converges as $n\to\infty$ to a random tree in $\Tdinf$. It is enough to show that 
the sequence $(\rT_n)_{n\ge0}$ is precompact in $\M{\Trinf}$.
Suppose indeed this holds. Let $\rT^*\in\M{\Trinf}$ be a limit point of the sequence $(\rT_n)_{n\ge 0}$.
By Lemma~\ref{lem:Trlinf_closed}, we actually have that $\rT^*\in\M{\Trlinf}$. By Lemma~\ref{lem:D_continuous}, $\Dop(\rT^*) = \lim \Dop(\rT_n)$. Lemma~\ref{lem:D_injective} now gives that $\rT^*$ is unique, whence $\rT_n \to \rT^*$ in law as $n \to \infty$, which was to be proven.

We now show precompactness of $(\rT_n)_{n\ge0}$ in $\M{\Trinf}$.
From the characterization of precompactness in Lemmata \ref{lem:precompact} and \ref{lem:precompact_2}, there are two points to show:
\begin{enumerate}
 \item For every $r\ge 0$, the family of random variables $\mu_n(\rT_n^{\le r})$ is tight.
 \item For every $r\ge 0$, there exist $R=R(r)$ and $n_0 = n_0(r)$, such that the family of random variables $(N_{r,R(r)}(\rT_n))_{r \geq 0,n\ge n_0(r)}$ is tight. 
\end{enumerate}
We prove both of these points by contradiction.

\emph{First point.} Assume that
there exists $t \geq 0$ such that the family of random variables $M_n := \mu_n(\rT_n^{\le r})$ is not tight. Restricting to a subsequence and conditioning on suitable events, we may assume that $\lim_{n\to\infty} M_n=+\infty$, in law. 
Under this assumption, we aim to show that there exists an integer $t'$ such that
\begin{equation}
\label{eq:Dt-not-tight}
\text{ the sequence } (\# \Dop(\rT_n)^{\le t'})_{n \in \N} \text{ is not tight.} 
\end{equation}
By Lemma~\ref{lem:precompact}, this will lead to the required contradiction.

Denote the set of vertices of the tree $\Dop(\rT_n^{\le t})$ by $U_n$ and set $N_n = \#U_n$.
We label the vertices in $U_n$ at random by elements of $\{1,\ldots ,N_n\}$, and define, for $k \in \N,$ the random variable $I_k$ by: $I_k = 1$ if $k \leq N_n$ and the vertex $k$ has more than $4t$ non-root ancestors (that is, ancestors distinct from the root) below, and $I_k= 0$ otherwise. The sum $N_n^I=\sum_{k} I_k$ then gives the number of those vertices that have more than $4t$ non-root ancestors. 
The number of non-root ancestors of a vertex in $U_n$ is, conditionally given $N_n$, dominated by a Bin($N_n-1,t/M_n$)-distributed random variable, since the length measure of the path from this vertex to $\rho$ in $\rT_n^{\le t}$ cannot exceed $t$ by definition.
We first compute the conditional expectation of $N_n^I$ given $N_n$:
\begin{align}
\label{eq:N_i-N_n}
\E(N_n^I|N_n) =  N_n \; \E(I_1|N_n) \le N_n \; \P\Big(\text{Bin}\Big(N_n-1, \frac{t}{M_n}\Big) \geq 4 \, t\,\Big|\,N_n\Big) \leq 
N_n \frac{(N_n-1)}{4 \, M_n}
\end{align}
using the conditional Markov inequality for the last estimate. Another application of the same inequality now gives:
\begin{align*}
\P(N_n^I > N_n/2)= \E(\P(N_n^I > N_n/2|N_n)) \le 2 \, \E\Big( \frac{\E(N_n^I|N_n)}{N_n} \Big) \le 1/2,
\end{align*}
using \eqref{eq:N_i-N_n}
and the fact that $N_n$ is a Poi($M_n$)-distributed random variable 
for the last inequality. Notice this inequality is equivalent to:
\begin{align*}
\P(N_n- N_n^I \ge N_n/2) = \P(N_n^I \le N_n/2) \ge 1/2.
\end{align*}
Now, fix an arbitrary $k \in \N$. Our assumption that $M_n$ diverges in law implies that the sequence $\P(N_n/2 \ge k)$ has limit 1, and in particular is larger than $3/4$ for $n$ large enough. This gives, for these values of $n$,
\begin{align*}
\P(N_n- N_n^I \ge k) 
&\ge \P(N_n- N_n^I \ge N_n/2, \, N_n/2 \ge k) \\
&\ge \P(N_n- N_n^I \ge N_n/2) + \P(N_n/2 \ge k) -1 \\ 
& \ge 1/4,
\end{align*}
and, together with the inequality $\# \Dop(\rT_n)^{\le t'} \geq N_n -N_n^{I}$ valid for $t'=4t+1$, this proves 
\eqref{eq:Dt-not-tight}.

\emph{Second point.} Assume that for all choices of  $R=R(r)$ and $n_0 = n_0(r)$ the family of random variables $(N_{r,R}(\rT_n))_{r\ge 0, n\ge n_0}$ is not tight. This means that there exists $c>0$, such that for every $N\in\N$, for some $r= r(N)$, 
\[
 \limsup_{n \to \infty} \P(N_{r,R}(\rT_n) \ge N) > c\quad\forall R\ge r,
\]
where we used the fact that  $N_{r,R}$ is decreasing in $R$. Since $N_{r,R}$ is also increasing in $r$, we may assume without loss of generality that $r(N)\to\infty$ as $N\to\infty$.
Under this assumption, we show that for large $N$,
\begin{equation}
\label{eq:numberofrays}
\limsup_{n \to \infty} \P(N_{2r,R/2}(\Dop[\rT_n]) \ge N/2) \geq c/2,\quad\forall R\ge 8r.
\end{equation}
By Proposition~\ref{prop:limit_Tdinf}, one readily checks that this implies that $\Dop(\rT_n)$ is not precompact in $\M{\Tdinf}$, which will yield the required contradiction.

In order to show \eqref{eq:numberofrays}, it is enough to show that for large $r$ and $N$, for every tree $\rT\in\Trlinf$,
\begin{equation}
\label{eq:numberofrays_2}
 \forall R\ge r: N_{r,R}(\rT) \ge N \quad\Longrightarrow \quad\forall R\ge 8r: \P(N_{2r,R/2}(\Dop[\rT]) \ge N/2) \geq 1/2.
\end{equation}
For this, fix $R\ge 8r$ and consider the vertices of $\rT$ at distance $r$ from the root that contribute to the quantity $\mathcal N := N_{r,R}(\rT)$. To each such vertex $v$, associate a single path (among possibly many) of length $R$ that links the root $\rho$ to a vertex at distance $R$, and contains~$v$. Let $F_v$ be the event that this path has less than $2r$ vertices in $V_0$ at distance $\le r$ from the root, but more than $R/2$ vertices in $V_0$ at distance $\le R$ from the root. For a vertex $v$, the probability $\P(F_v)$ is independent of $v$, and arbitrary close to 1 for large $r$, so in particular larger than~$3/4$ for $r$ large enough. 

Now let $\mathcal N^{\bar F}$ be the number of vertices $v$ contributing to $\mathcal N$ that satisfy the complementary event $\bar F_v$. Then
$\E(\mathcal N^{\bar F}|\mathcal N) = \mathcal N (1-\P(F_v))$, and by the Markov inequality,  for $\mathcal N$ and $r$ large enough,
$$\P(\mathcal N^{\bar F} \geq \mathcal N/2 ) \le 2 (1-\P(F_v)) \leq 1/2
\text{ or, equivalently, }
\P(\mathcal N-\mathcal N^{\bar F} > \mathcal N/2 ) \ge 1/2.$$
Together with the pointwise inequality $N_{2r,R/2} (\Dop[\rT_n]) \geq \mathcal N-\mathcal N^{\bar F}$ this proves \eqref{eq:numberofrays_2} and finishes the proof.
\end{proof}

\section{Tree rescaling: proofs of Theorems~\ref{th:self_similar} and \ref{th:inverse_limit}}
\label{sec:4}

In this section, we will prove a coupling lemma (Lemma~\ref{lem:shrink_coupling} below), which, together with Theorem~\ref{th:topologies}, will yield Theorems~\ref{th:inverse_limit} and \ref{th:self_similar}.

The space of discrete trees $\Tdinf$ is naturally embedded into the space 
$\Trlinf$ via the following embedding $\iota$: given $T \in \Tdinf$, we define $\tembed{T}=(\rV, d, 
\root, \mu)$ as follows:
\begin{itemize}
\item The set of vertices is given by 
$\rV= \{ (i,x), i \in  T \, \setminus \, \{\root\}, x \in [0,1) \} \cup \{(\root,0)\}$.
\item The distance is defined by $d((i,x),(j,y)) = d_T(i,j)- x - y$ if 
($i \not\preceq_T j$ and $j \not\preceq_T i$) 
and $d((i,x),(j,y)) = d_T(i,j)- x + y$ if $j \preceq_T i$, and it is symmetric in its arguments.
\item $\mu$ is the Lebesgue measure $\ell_{\rT}$.
\end{itemize}
Informally, the tree $\tembed{T}$ is defined from $T$ 
by adding segments of length 1 between the vertices of $T$, 
and the element $(i,x)$ of $\rV$ is at distance $x$ of $(i,0)$ on the path $\llbracket \root,i \rrbracket$ to the root.

\begin{lemma}
\label{lem:shrink_coupling}
Let $(p_n)_{n\ge0}$ and $(q_n)_{n\ge0}$ be sequences of positive numbers such that $p_n\to 0$ and $q_n\to 0$ as $n\to\infty$. Let $T_n\in\M{\Tdinf}$ be a sequence of random rooted trees. Then the following statements are equivalent:
\begin{enumerate}
 \item The sequence $(\Sop[p_n,q_n](T_n))_{n\ge 0}$ is tight in\footnote{Here, and below, we say that a family of random trees on $\Tdinf$ is \emph{tight in $\Tdinfall$} if the family of their laws, seen as laws on $\Tdinfall$, are tight.} $\Tdinfall$.
 \item The sequence $(\Dop[\Srop[p_n,q_n](\tembed{T_n})])_{n\ge 0}$ is tight in $\Tdinfall$.
\end{enumerate}
In this case, for every $m\ge 1$ and $\ep>0$, there exists an integer $n_0$ and for every $n\ge n_0$ a coupling between $\Sop[p_n,q_n](T_n)$ and $\Dop[\Srop[p_n,q_n](\tembed{T_n})]$, such that
\begin{equation}
 \label{eq:coupling}
 \P\left(\left[\Sop[p_n,q_n](T_n)\right]^{\le m} = \left[\Dop[\Srop[p_n,q_n](\tembed{T_n})]\right]^{\le m}\right) > 1-\ep.
\end{equation}
\end{lemma}
\begin{proof}
Recall that by definition of the topology of local convergence on $\Tdinf$, a sequence of random discrete trees $\widetilde T_n$ in $\M{\Tdinf}$ is tight if and only if for every $m\ge 1$, the sequence $(\#V(\widetilde T_n^{\le m}))_{n\ge1}$ is tight. In order to apply this, we first need to define a suitable coupling between the operations $\Sop$ and $\Srop$ and suitable one-dimensional stochastic processes.

Let $T\in\Tdinf$. Recall that in the construction of $\Sop(T)$, the vertices are colored with two colors, say black and white; the black vertices (and the root) are retained, whereas the white vertices are discarded. Given a tree $T\in\Tdinf$ we can couple this operation with two infinite sequences $(B_i)_{i\ge 1}$ and $(B'_i)_{i\ge 1}$ of iid random variables distributed according to the Bernoulli distribution with parameter $p$ and $q$, respectively: we perform a 
breadth-first traversal
of the vertices of the tree starting from the root and color the vertex visited at step $i$ according to $B_i$ (1 = black and 0 = white) if it is an off-spine vertex, and according to $B'_i$ if it is a spine vertex. Given an integer $m\ge 1$, we add an additional rule: When a black vertex is visited which has exactly 
$m$ black ancestors, then 
the subtree above it
is subsequently ignored by the algorithm. The restriction of the tree $T$ to the black vertices then exactly gives the tree $\Sop(T)^{\le m}$. Note that this algorithm terminates almost surely since $\Sop(T)$ is locally finite by definition, hence $\Sop(T)^{\le m}$ is finite and the breadth-first traversal only has to go to a certain (random, but finite) depth of the tree $T$.

In order to construct the tree $\Dop[\Srop(\tembed{T})]^{\le m}$ one can proceed in a similar manner, but using now two Poisson processes $(P_t)_{t\ge 0}$ and $(P'_t)_{t\ge 0}$ with parameters $p$ and $q$ (respectively) defined on the $\R$-tree $\tembed{T}$. We omit the details.

For every $M\ge 0$, we now couple the sequences $(B_i)_{i\ge 1}$ and $(B'_i)_{i\ge 1}$ with the Poisson processes $(P_t)_{t\ge 0}$ and $(P'_t)_{t\ge 0}$ in such a way that with probability $1-O(M(p\vee q))$, for every $i\le M/p$, $B_i = 1$ if and only if $P_i-P_{i-1} = 1$ and for every $i\le M/q$, $B'_i = 1$ if and only if $P'_i-P'_{i-1} = 1$. It is now easy to check that for large $M$ and small $p$ and $q$, on the event that (either) $\# V(\Sop(T)^{\le m}) \le M/2$ or $\# V(\Dop[\Srop(\tembed{T})]^{\le m}) \le M/2$, both trees agree with high probability. This implies that if the first or second statement of the lemma holds, then \eqref{eq:coupling} is true, which in turn implies equivalence of the two statements. This finishes the proof of the lemma.
\end{proof}

\begin{proof}[Proof of Theorem~\ref{th:inverse_limit}]
We start with the easy direction: Let  $\rT \in \M{\Trlone}$. 
The family 
$T_n=\Dop[\rT][n]$
is a compatible family of random rooted trees since, for $n \geq m \geq 1$:
\[
\Cop{T_n}{m} = \Cop{\Dop[\rT][n]}{m} = \Dop[\rT][m]=  T_m.
\]
Now let $(T_n)_{n\in\N}$ be a compatible family of random trees and let $(B_n)_{n\in\N}$ be a sequence of Bin($n,1/n$)-distributed random variables, independent of $(T_n)_{n\in\N}$. 
We then have the following equality,
\[\Cop{T_n}{B_n} =  T_{B_n}.\]  
On the RHS, there is convergence towards $T_{B}$ as $n \to \infty$, where $B$ is a Poi(1)-distributed random variable independent of $(T_n)_{n\in\N}$. On the LHS, we have
$\Cop{T_n}{B_n}=\Sop[1/n](T_n)$, and by
 Lemma~\ref{lem:shrink_coupling} there exists a coupling such that the equality
\[
\Sop[1/n](T_n)=\Dop[\Srop[1/n](\tembed{T_n})]
\]
holds with high probability as $n \to \infty$. 
The space  $\Trlone$ is compact by Lemma~\ref{lem:Trlone_compact}, therefore the sequence 
$(\Srop[1/n](\tembed{T_n}))_{n\ge1} \in \M{\Trlone}$ is tight.
Since the map $\Dop$ is continuous and injective by Theorem~\ref{th:topologies},
every subsequential limit $\rT$ has $T_{B}=\Dop[\rT]$ and is thus unique. Furthermore, conditioning on $B = m$ yields $T_m = \Dop[\rT][m]$ for each $m\in\N$.
\end{proof}

\begin{proof}[Proof of Theorem~\ref{th:self_similar}]
One direction is obvious: Let  $\rT \in \M{\Trlinf}$ be a $(p,q)$-self-similar random $\R$-tree. Lemma~\ref{lem:commutation} then yields
\[
 \Sop(\Dop[\rT]) = \Dop(\Srop(\rT)) = \Dop(\rT),
\]
whence the discrete tree $\Dop[\rT]$ is $(p,q)$-self-similar as well. 

For the other direction, let $T\in\M{\Tdinf}$ be a  $(p,q)$-self-similar random rooted tree, i.e.\ $\Sop(T) = T$. Iterating this equality yields for each integer $n\ge 1$,
\begin{equation}
\label{eq:th2-1}
\Sop[p^n,q^n](T) = \Sop^n(T)=T.
\end{equation}
In particular, the sequence $(\Sop[p^n,q^n](T))_{n\ge1}$ is tight in $\Tdinfall$. Lemma~\ref{lem:shrink_coupling} now yields that the sequence $(\Dop[\Srop[p^n,q^n](\tembed{T})])_{n\ge1}$ is tight in $\Tdinfall$ and that every subsequential limit equals $T$. 
By Theorem~\ref{th:topologies}, the sequence $(\Srop[p^n,q^n](\tembed{T}))_{n\ge1}$ then converges in law to a random tree $\rT$ taking values in $\Trlinf$ and such that $\Dop(\rT) = T$. This uniquely determines the tree $\rT$ by injectivity of the map $\Dop$. Furthermore, by Lemma~\ref{lem:commutation} and the self-similarity of $T$, we have $\Dop(\Srop(\rT)) = \Sop(T) = T = \Dop(\rT)$, whence, again by injectivity of $\Dop$, $\Srop(\rT) = \rT$. This finishes the proof of the theorem.
\end{proof}

\section{Examples of self-similar trees}
\label{sec:5}

In this section, we construct some examples of $(p,q)$-self-similar $\R$-trees. 
We do not believe it is possible to completely characterize this family, similarly to the situation for self-similar real-valued processes. For simplicity, we restrict ourselves to trees whose spine consists of a single infinite ray only.

\begin{enumerate}
 \item \textbf{Subordination of a real-valued self-similar process.} If $(\rT,d,\root,\mu)$ is a $(p,q)$-self-similar $\R$-tree, then let $\rV_t$ be the subset consisting of the vertices whose most recent ancestor on the spine is at distance at most $t$ from the root. Setting  
 \begin{equation}
\label{eq:Xdef}
X(t) = \mu(\rV_t) - t
\end{equation}
defines a (semi)-self-similar real-valued non-decreasing process with Hurst exponent\footnote{The exponent $H$ is called the Hurst exponent after the study in hydrology \cite{HBS65}, see also the paper \cite{MV68}.}  $H=\log p/\log q$, i.e. 
\begin{equation}
\label{eq:Xss}
(q^{-H}X(qt),t\ge0)\stackrel{\text{law}}{=}(X(t),t\ge0).
\end{equation}
On the other hand, if we are given  such a process $X(t)$, we can construct from it a  $(p,q)$-self-similar $\R$-tree $\rT \in \M{\Trlinf}$. Write $X_c(t)$ for its continuous part and $X_j(t)$ for its jump part. Let $\rT' \in \M{\Trlone}$ be an arbitrary random tree. The random tree $\rT$ is then constructed as follows:
\begin{itemize}
	\item $\Sp(\rT)$ consists of a single infinite ray,
	and $\mu(\Sp(\rT)\cap[0,t]) = X_c(t)+t$,
	\item For every jump time $t$ of $X_j$, we attach an independent copy of $\rT'$ to the spine at distance $t$ of the root, rescaled by the size of the jump $X_j(t)-X_j(t-)$.
\end{itemize}
It is easy to show that the resulting tree is indeed $(p,q)$-self-similar.

 One can easily generalize the above construction. For example, instead of attaching independent rescaled copies of the same tree $\rT'$ to the spine, one can take a $|\log q|$-stationary process $(\rT'(s),s\in\R)$ of rooted, probability-measured $\R$-trees ($|\log q|$-stationary means that  $(\rT(s),s\in\R)$ is equal in law to  $(\rT(s+|\log q|),s\in\R)$), and attach a rescaled copy of $\rT'(\log t)$ at the point $t$ on the spine. One can also introduce a stronger dependency between the process $X(t)$ and the trees. For example, let $X(t)$ be as above and suppose for simplicity that it is a pure-jump process. Let $\mathcal R$ be the set of its \emph{record jumps}, i.e. $r\in \mathcal R$ if and only if $X(r)-X(r-) > X(s)-X(s-)$ for all $s<r$. Write $\R = \{\ldots < r_{-1}<r_0<r_1<\ldots\}$, with $r_0 \le 1 < r_1$ and let $\rT'_n$ be a sequence of iid copies of a probability-measured rooted $\R$-tree. We then construct a $(p,q)$-self-similar tree as follows: For every $t$, let $N(t)$ be such that $r_{N(t)} \le t < r_{N(t)+1}$. Then for each jump time $t$ of $X(t)$, add the tree $\rT'_{N(t)}$ to the spine, rescaled by the size $X(t)-X(t-)$ of the jump. One readily checks that the resulting process is $(p,q)$-self-similar.

\item \textbf{Rescaling along the spine.} Given a $(p,q)$-self-similar $\R$-tree $\rT$, one can easily construct a whole family of self-similar $\R$-trees: Let $\beta > 0$. First, one can rescale the tree along the spine: define a new $\R$-tree $\rT'$ obtained from $\rT$ by mapping a point $t$ on the spine to $t^\beta$. The mass process of this new tree is $X^\beta(t) = X(t^{1/\beta})$, and therefore $(p^{-1}X^\beta(q^\beta t);t\ge0) \stackrel{\text{law}}{=} (X^\beta (t);t\ge0)$. Since we have not changed the structure of the subtrees, it follows that the resulting tree is $(p,q^\beta)$-self-similar.

For $\gamma\in\R$, one can also define a new mass process by setting $X^\gamma(t) = \int_0^t s^{\gamma}dX(s)$, as long as this quantity is finite for some (hence, any) $t>0$. The tree defined by this process in the canonical way (i.e.\ by rescaling the subtrees of the spine and the measure $\mu$ on the spine) is then $(pq^\gamma,q)$-self-similar.

Finally, one can apply the previous scaling to the continuous part $X_c(t)$ of the mass process only, and scale the jump process $X_j(t)$ instead by setting $X_j^\delta(t) = \sum_{s\le t} (X_j(s)-X_j(s-))^\delta$ for some $\delta > 0$. If $\delta$ and $\gamma$ are such that $pq^\gamma = p^\delta$, then the mass process $X^{\gamma,\delta}(t) = X_c^\gamma(t) + X_j^\delta(t)$, if it exists, defines a $(p^\delta,q)$-self-similar tree in the canonical way.
\end{enumerate}

\section{Translation invariant self-similar trees}
\label{sec:6}

In this section, we study self-similar trees which are invariant under translation along the spine. For simplicity, we restrict ourselves to one-ended self-similar trees. We denote the corresponding subspaces of $\Tdinf$ and $\Trlinf$ by $\Tdspine$ and $\Trspine$, respectively. Throughout the section, $T\in\M{\Tdspine}$ will denote a one-ended $(p,q)$-self-similar tree and $\rT = (\rV,d,\root,\mu)\in\M{\Trspine}$ the limiting $\R$-tree obtained from Theorem~\ref{th:self_similar}. We define a shift operator $\S:\M{\Tdspine}\to\M{\Tdspine}$ which maps $T$ to the subtree rooted at the vertex on the spine at distance 1 from the root. 
We say $T$ is \emph{translation invariant}, if $\S T \stackrel{\text{law}}{=} T$. Similarly, for $t\ge0$, we define a shift operator $\S_t:\M{\Trspine}\to\M{\Trspine}$ mapping $\rT$ to its subtree rooted at the vertex on the spine at distance $t$ from the root. Note that $(\S_t)_{t\ge0}$ is a semigroup. We then say that $\rT$ is translation invariant if $\S_t\rT  = \rT$ for every $t\ge0$. The following proposition says that the two notions are equivalent:

\begin{proposition}
\label{prop:translation}
$T$ is translation invariant if and only if $\rT$ is translation invariant.
\end{proposition}

\begin{lemma}
 \label{lem:shift_continuous}
 The semigroup $(\S_t)_{t\ge0}$ is strongly left-continuous, i.e.\ for every $\rT\in\M{\Trspine}$, the function $t\mapsto \S_t(\rT)$ is left-continuous.
\end{lemma}

\begin{proof}
It is enough to check the lemma for a deterministic tree $\rT \in \Trspine$.
Consider the subtrees $\rT_1=\S_{t-\ep}(\rT)$ and $\rT_2=\S_t(\rT)$ rooted at the spine vertices at distance $t-\ep$ and $t$ from $\rho$, respectively denoted by $\rho_1$ and $\rho_2$. From the definition  \eqref{eq:GHPinf} of the GHP distance and the dominated convergence Theorem, 
it is enough to prove that the GHP distance between the compact trees $\rT_1^{\leq r}$ and $\rT_2^{\leq r}$ has a null $\ep \to 0$ limit for any fixed  value of $r>0$. We let $\mu_1$ and $\mu_2$ be the measures associated with $\rT_1^{\leq r}$ and $\rT_2^{\leq r}$ respectively.
We have $d(\root_1,\root_2) = \ep$, and, with $\rV^{<r}$ the restriction of $\rV$ to the \textit{open} ball of radius $r$ centered at $\rho$, it holds
$$d_{P}^c(\mu_1,\mu_2) \leq \mu(\rV^{<t} \, \backslash \, \rV^{< t- \ep}) +
\mu(\rV^{<t+r} \, \backslash \, \rV^{< t+r-\ep}) $$
and the last expression has a null $\ep \to 0$ limit since the two sets on the RHS decrease to the null set. These elements combined with Lemma \ref{lem:GP_GHP} now allow to conclude since: $$d_{GHP}^c(\rT_1^{\leq r},\rT_2^{\leq r}) \leq 3 \, d_{GP}^c(\rT_1^{\leq r},\rT_2^{\leq r}) \leq d(\root_1,\root_2)+ d_{P}^c(\mu_1,\mu_2).$$ 
\end{proof}

\begin{proof}[Proof of Proposition~\ref{prop:translation}]
By Theorem~\ref{th:self_similar}, the tree $T$ is obtained from $\rT$ by sampling vertices according to a Poisson process on $\rT$ with intensity $\mu$. By definition of the sampling procedure, the vertices on the spine are those which are sampled according to $\ell_\Sp$, where $\ell_\Sp = \ell_{\rT}|_{\Sp(\rT)}$ is the restriction of the length measure $\ell_\rT$ to the spine. It follows that 
\begin{equation}
\label{eq:SnT}
\S^nT = \Dop[\S_{E_1+\ldots+E_n}\rT] ,
\end{equation}
where $E_1,\ldots,E_n$ are independent exponential random variables with parameter 1, independent from $\rT$. Setting $n=1$ in \eqref{eq:SnT} directly yields the ``if'' statement of the proposition.

Now suppose that $T$ is translation invariant. Then, for every $n\ge 0$ and $k\ge0$, we have
\begin{align*}
 \Dop[\rT] &= T && \text{by hypothesis on $T$ and $\rT$}\\
 &= \Sop^k T && \text{by the self-similarity of $T$}\\
 &= \Sop^k \S^nT && \text{by translation invariance of $T$}\\
 &= \Sop^k \Dop[\S_{E_1+\ldots+E_n}\rT] && \text{by \eqref{eq:SnT}}\\
 &= \Dop[\Srop^k \S_{E_1+\ldots+E_n}\rT] && \text{by Lemma~\ref{lem:commutation}}\\
 &= \Dop[\S_{q^k E_1+\ldots+q^k E_n} \Srop^k\rT] && \text{by definition}\\
 &= \Dop[\S_{q^k E_1+\ldots+q^k E_n} \rT] && \text{by the self-similarity of $\rT$}.
\end{align*}
Since $\Dop$ is injective by Lemma~\ref{lem:D_injective}, this implies for every $n\ge 0$ and $k\ge0$,
\[
 \rT = \S_{q^k E_1+\ldots+q^k E_n} \rT.
\]
By the law of large numbers, for any $t\ge 0$, we can now let $k,n\to\infty$ in such a way that $q^k E_1+\ldots+q^k E_n$ converges from the left to $t$ almost surely. By Lemma~\ref{lem:shift_continuous}, this yields the statement.
\end{proof}

We now show that for translation invariant trees, the range of possible values for $p$ and $q$ is restricted. The mass process $X(t)$ defined in \ref{eq:Xdef} is real-valued, increasing, and semi-self-similar process with Hurst exponent $H=\log p/\log q$. Recall $X_c(t)$ and $X_j(t)$ denote respectively the continuous part and the jump part of $X(t)$, and note that $X_c(t) = \mu(\Sp(\rT)\cap[0,t])-t$.

\begin{proposition}
\label{prop:ti}
Suppose $\rT\in\M{\Trspine}$ is $(p,q)$-self-similar, translation invariant and non-degenerate (i.e.\ it is not isometric to $\R_+$). Then $q\ge p$. Moreover, if $q>p$, then $X_c \equiv 0$ almost surely, and if $q=p$, then $X_j \equiv 0$ almost surely and $X_c(t) = X(1) \, t$ for every $t\ge 0$.
\end{proposition}
\begin{proof}
This proposition is essentially a corollary of results of Vervaat \cite{Vervaat1985}.
$X$ satisfies the hypotheses 1.4 in \cite{Vervaat1985}, except for one: 
Equation \eqref{eq:Xss} does not hold for \emph{every} $p,q>0$ with $\log p/\log q = H$ in our case. 
This is the only missing assumption in Theorem~3.3 in \cite{Vervaat1985}. We claim that the conclusion of that theorem
still hold under \eqref{eq:Xss}.
To justify this claim, we observe that the proof of Theorem 3.3 relies on Theorem 3.1, items a and b for $H \le 1$, and on Theorem 3.5 for $H >1$. The proof of the latter is not affected by our assumption. We therefore have to verify that the conclusions of Vervaat's Theorem 3.1, items a and b still hold under \eqref{eq:Xss}.

Let $\mathcal I$ be the invariant $\sigma$-field of the stationary sequence $(X(t)-X(t-1), t \in \N)$.
In the proof of Theorem 3.1, Vervaat establishes the following for a \emph{truly} semi-self-similar process:
\begin{multline*}
t^{H-1} X(1) \stackrel{\text{law}}{=} 
\frac{X(t)}{t} = \frac{X(t)-X(0)}{t}  = \frac{1}{t} \sum_{1 \leq s \leq t} X(s)-X(s-1)\\
\to \E(X(1) -X(0) \, | \, \mathcal I)
=\E(X(1) \, | \, \mathcal I),
\end{multline*}
using respectively semi-self-similarity, $X(0)=0$ a.s., and
Birkhoff's ergodic theorem in a version that allows for infinite means, like the one presented in Theorem 3.7 of \cite{Vervaat1985}.
Because the integer $t$ is not necessarily of the form $q^{-m}$, the left-most equality has to be replaced in our case by:
$$	
\frac{X(q^m t)}{p^m t} \stackrel{\text{law}}{=} 
\frac{X(t)}{t} \cdot
$$
Now, the sequence of integers $t=t(m)=  \lceil q^{-m}\rceil$ satisfies $q^m t \geq 1$ and $q^m t \to 1$ as $m \to \infty$. Since the process $X$ has càdlàg sample paths (this may be seen using arguments similar as in the proof of Lemma \ref{lem:shift_continuous}), $X(q^m t)$ converges to $X(1)$ a.s. as $m \to \infty$.
From that point on, the proof of Theorem 3.1, items a and b, follows unchanged.

Therefore the conclusion of Theorem 3.3 holds, and that conclusion implies the statement of our proposition, since the event $A$ (in that theorem) that $X$ has locally bounded variation has probability 1 in our case due to the monotonicity of the sample paths of $X$. 
\end{proof}

We now want to study the $(p,q)$-self-similar discrete trees $T$ which are translation invariant \emph{and} for which the subtrees along the spine are independent (hence, iid).

\begin{proposition}
\label{prop:iid}
The subtrees of $T$ along the spine are iid if and only if $q\ge p$ and $\rT$ is constructed as follows:
\begin{itemize}
	\item If $q>p$, then $\mu_{|\,\Sp}= \ell_{\R_+}$, and the point process with atoms $(t,\rT_t)$, where $\rT_t$ is the subtree rooted at the spine vertex at distance $t$ from the root, is a Poisson point process with intensity $dt\otimes \nu(d\rT)$, with the measure $\nu$ decomposing as follows: There exists a measurable family $(\sigma_x)_{x >0}$ of probability measures on the space
	$\Trlone$
	such that $\sigma_x = \sigma_{px}$ for every $x$, and a measure $\Lambda(dx)$ on $(0,\infty)$ satisfying 
	$\Lambda(A) = q \Lambda(pA)$ for every Borel $A\subset(0,\infty)$,
	such that $\nu$ decomposes as the semi-direct product $\nu = \Lambda(dx) \sigma_x^x$, where $\sigma_x^x$ is the push-forward of the measure $\sigma_x$ under the map $\Srop[x]$.
	\item If $q=p$, then $\rT$ is the non-negative real line with a (deterministic) multiple of Lebesgue measure:
	$\rV=\R_+$, $\rho=0$ and there exists $\lambda \geq 1$ such that $\mu= \lambda \, \ell_{\R_+}.$
\end{itemize}
\end{proposition}
\begin{proof}
By Proposition~\ref{prop:ti}, we either have $q>p$ and the mass process $X(t)$ is a pure-jump process, or $q=p$ and the mass process is continuous with $X(t) = tX(1)$. The statement in the case $q=p$ now follows from the fact that the subtrees of $T$ are independent if and only if $X(1)$ is a deterministic constant.

Now consider the case $q>p$. By Proposition~\ref{prop:ti}, the mass process $X(t)$ is then a pure-jump process, i.e.\ the restriction of $\mu$ to the spine equals the length measure. Denote by $T^i$ for every $i\ge0$ the off-spine subtree of the $i$-th vertex on the spine of $T$. Furthermore, for $s\le t$, denote by $\rT^{s,t}$ the \emph{concatenation} of the off-spine subtrees of the tree $\rT$ rooted at vertices $v$ on the spine with $s\le d(0,v) < t$. Here, the \emph{concatenation} of a collection of rooted trees is defined to be the rooted tree obtained from the disjoint union of the trees by identification of the roots. By the definition of the discretization operation $\Dop$, we then have
\begin{equation}
 \label{eq:subtrees}
 (T^0,T^1,T^2,\ldots) \stackrel{\text{law}}{=} (\Dop(\rT^{\xi_0,\xi_1}),\Dop(\rT^{\xi_1,\xi_2}),\ldots),
\end{equation}
where $\xi_0=0$ and for $n \in \N^*$, $\xi_n = \sum_{k=1}^n E_k$, with $E_1,E_2,\ldots$ a sequence of independent exponential random variables with parameter 1, independent from $\rT$. Equation \eqref{eq:subtrees} now shows that if the point process with atoms $(t,\rT_t)$ is a translation invariant Poisson process, then the trees $(T^i)_{i\ge 0}$ are iid. 

Now assume that the trees $(T^i)_{i\ge 0}$ are iid. 
Fix $0 = t_0 < t_1 < t_2 < \ldots$. We first notice that:
\begin{align}
 \Dop(\rT^{t_0,t_1}) & \stackrel{\text{law}}{=}  \Dop(\Srop^{k}(\rT)^{t_0,t_1}) && \text{by the self-similarity of $\rT$} \nonumber \\ 
 &= \Dop(\Srop^{k}(\rT^{t_0 q^{-k}, t_1 q^{-k}})) && \text{by definition} \nonumber \\
 &\stackrel{\text{law}}{=}  \Sop^k(\Dop(\rT^{t_0 q^{-k}, t_1 q^{-k}})) && \text{by Lemma~\ref{lem:commutation}.} \label{eq:subtrees2}
\end{align}
Fix an integer $i$, and real numbers $t_1', \ldots, t_i'$ satisfying $t_{j-1} < t_j' < t_j$ 
for every $j \in \{1,\ldots i\}$. Let $j$ be in this set. By the law of large numbers, we may find integers $n_j(k)$  such that, for $k$ large enough, 
\begin{equation}
 \label{eq:subtrees3}
 t_j' q^{-k} < \xi_{n_j}  <  t_j q^{-k},
 \end{equation}
with the $(\xi_n)$ distributed as above. Set also $n_0=0$. Applying \eqref{eq:subtrees2} with $t_0$ and $t_1$ replaced by $t_{j-1}$ and $t'_{j}$, and recalling \eqref{eq:subtrees3} and  \eqref{eq:subtrees}, we deduce that
that the discretized tree $\Dop(\rT^{t_{j-1},t_{j}'})$ is  a function of $T^{n_{j-1}},...,T^{n_j}$, for $j \in \{1,\ldots,i\}$. Therefore the collection of trees
\(
 (\rT^{t_0,t_1'},\ldots,\rT^{t_{i-1},t_i'})
\)
is independent. The numbers $t_1', \ldots, t_i'$ being arbitrary, this implies the collection  \(
 (\rT^{t_0,t_1},\ldots,\rT^{t_{i-1},t_i})
\) is independent, therefore the point process $(t,\rT_t)$ is a Poisson process. By translation invariance, its intensity measure is of the form $dt \otimes \nu$ for some measure $\nu$. By the $(p,q)$-self-similarity, $\nu = q^{-1} \nu^p$, where $\nu^p$ is the push-forward of the measure under the map $\rT \mapsto p\rT$. Disintegrating the measure $\nu$ with respect to the mass of the tree yields the decomposition stated in the theorem.
\end{proof}

In case $q=p$ and the subtrees of $T$ along the spine are iid, their common distribution is that of a rooted tree with a $\operatorname{Geo}(1/\lambda)$-distributed number \footnote{Our geometric distribution starts at $0$,  $\P(\operatorname{Geo}(\gamma)=k)=(1-\gamma)^k \gamma$, for each $k \in \N$, for $\gamma \in (0,1]$.} of edges adjacent to the root.
In case $q>p$, there is the following corollary that follows from Proposition~\ref{prop:iid} and standard properties of Poisson processes. Recall that we define the \emph{concatenation} of a collection of rooted trees to be the rooted tree obtained from the disjoint union of the trees by identification of the roots.

\begin{corollary}
\label{cor:iid}
In case $q>p$ and the subtrees of $T$ along the spine are iid, their distribution is characterized as follows:
There exists a constant $\gamma\in(0,1]$, as well as 
a measurable family of probability measures $\sigma_x$ and a measure $\Lambda(dx)$ as in the statement of Proposition~\ref{prop:iid}, such that
\begin{itemize}
\item $T^0$ is the concatenation of copies of $\operatorname{Geo}(c)$-distributed number of independent copies of a tree $T'$, where $c= 1/(1+d)$ and $d=\int_0^\infty (1-e^{-x})\Lambda(dx)$, and
\item $\E[F(T')] = d^{-1} \int_0^\infty \Lambda(dx) \int \sigma_x(d\rT) \E[F(\Dop(\rT,\Poi(x)))]$ for every bounded measurable function $F$ with $F(\root)=0$.
\end{itemize}
\end{corollary}
If a measure $\Lambda(dx)$ satisfies the condition in Proposition~\ref{prop:iid}, any measure proportional to $\Lambda(dx)$ again satisfies this condition, so the parameter $c$ of the Geometric random variable in Corollary~ \ref{cor:iid} may indeed take arbitrary values in $(0,1]$. 
\begin{remark}
\label{rem:qsd}
The number of edges in the tree $T'$ that appears in the statement of Corollary~\ref{cor:iid} follows the law 
\[
\P[N(T')=k] = d^{-1} \int_0^\infty \P(\Poi(x) = k) \Lambda(dx), \; \forall k\ge 1.
\]
The laws of this form with $d$ and $\Lambda$ as in the statement of Proposition~\ref{prop:iid} are exactly the quasi-stationary distributions of the Markov chain 
$(Z_{(\log p^{-1}) n};n=0,1,2,\ldots)$,
where $(Z_t;t\ge0)$ is the standard pure death process (i.e., $Z_{t-} \to Z_t-1$ with rate $Z_{t-}$) killed at 0 \cite{M15}.
If $P$ denotes the substochastic transition matrix (on $\N^*$) of this discrete time Markov chain, they are the distributions $\eta$ on $\N^*$ that satisfy:
$$\eta \, P =  q \eta.$$
\end{remark}

\section{A different approach to Theorem~\ref{th:inverse_limit}}
\label{sec:7}

Another approach may be considered to prove our Theorems. 
This short section gives an intuition on the objects that are introduced in the introduction. Our reasoning is based on Theorem~\ref{th:inverse_limit}. Two proof strategies for this theorem are presented and then compared.



Consider a sequence of random trees $(T_n)_{n\in\N} \in \M{\Tdfin}$ satisfying $T_m \stackrel{\text{law}}{=} \Cop{T_n}{m}$ for $n \geq m \geq 1$.

A powerful idea, popularized by David Aldous \cite{Aldous1993}, consists in using exchangeability to construct concrete representations of the (\emph{a priori} abstractly defined) inverse (or projective) limit of such a sequence\footnote{For examples of other settings see e.g.  \cite{Janson2011} and \cite{Haulk2011} and the references therein.}.

\begin{itemize}
\item  
Attach labels to the vertices of $T_n$ according to an independent random permutation of $\{1,\ldots,n\}$, 
and consider the random partial order on $\{1,\ldots,n\}$ induced by the ancestral relation $\preceq_{T_n}$ on the vertices of $T_n$.
From the compatibility of the family $(T_n, n \geq 1)$ and the Kolmogorov's extension theorem, these orders extend to an exchangeable partial order $\preceq$ on 
$\N^*:= \{1,2,\ldots\}$, whose restriction to $\{1,\ldots,n\}$ for each $n$ is $\preceq_{T_n}$. 
For any two integers $i$ and $j$ with most recent common ancestor   $i \wedge j$ (for $\preceq_{T_n}$), de Finetti's theorem ensures that $$d(i,j) := \lim_{n \to \infty} \frac{1}{n} \, \sum_{k \leq n, k \neq i,j} \ind_{k \preceq i, k \preceq j, k \not\preceq i \wedge j} \text{ exists a.s.} $$ Also $(d(i,j), i,j \in \N^\star)$ defines a random pseudo-metric on $\N^*$. Its metric completion is a rooted $\R$-tree. A natural probability measure attached to that $\R$-tree is:
$$\mu( \llbracket i,j \rrbracket ) :=  \lim_{n \to \infty} \frac{1}{n} \sum_{k \leq n, k \neq i,j} \ind_{d(i,k)+d(k,j)=d(i,j)}.$$
By definition, the measure $\mu$ dominates the length measure: $\mu( \llbracket i,j \rrbracket ) \ge d(i,j)$. 

This construction is an example of the idea, popularized by David Aldous \cite{Aldous1993}, of using exchangeability to construct concrete representations of (\emph{a priori} abstractly defined) inverse (or projective) limit of compatible sequences\footnote{For examples of other settings see e.g.  \cite{Janson2011} and \cite{Haulk2015} and the references therein.}. 
\item
Take $m=M$ an independent Poisson random variable $M$ with unit parameter in the compatibility relation above: this gives $T_M \stackrel{\text{law}}{=} \Cop{T_n}{M}$. Also, one can consider $\rT_n=\Srop[1/n](\tembed{T_n})$, the random $\R$-tree obtained from $T_n$ by giving its edges an equal length, $1/n$. The discretization of that $\R$-tree is $\Dop(\rT_n)$.
A comparison of the constructions of of $\mathcal C(T_n,M)$ and  $\Dop(\rT_n)$ now suggests that
for a large $n$, the two trees should be close.
But the distribution of the sequence  $(\Cop{T_n}{M}, n \geq M)$ is by definition independent of $n$. Therefore the sequence $\rT_n$ converges in the topology induced by the map $\Dop$. This leads us to the two key questions of this work, namely - the identification of the topology induced by the map $\Dop$,and - the identification of its relatively compacts subsets. The answers we provide to these questions entail the existence of a random rooted measured $\R$-tree $\rT\in \M{\Trlone}$ satisfying $T_n= \Dop(\rT,n)$.
\end{itemize}

The first construction is perhaps more natural than the second one, also, it quickly point to the right objects to be introduced, like $\Trlone$. However, a more careful study reveals technical measurability issues with the first approach, that do not arise with the second one. The second approach, although it seems a priori more difficult to implement, had our favour.


\appendix

\section{Tree spaces}

In this appendix, we collect some properties of the tree spaces that we work with in this paper, namely, the space of locally finite (graph-theoretic) trees and the space of locally compact measured $\R$-trees.

\subsection{Discrete trees}
\label{sec:Dtrees}

In this section, we consider rooted, locally finite trees $T=(V,E,\root)$ in the graph-theoretic sense. We recall from the introduction that two such trees are called \emph{equivalent} if there exists a root-preserving graph isomorphism between them. For simplicity, we will always identify an equivalence class with its representatives. We then denote by $\Tdinfall$ the space of (equivalence classes of) trees 
and by $\Tdfin\subset \Tdinfall$ the subspace of finite trees, both endowed with the topology of local convergence (see introduction). Also recall that a tree $T\in\Tdinfall$ determines and is determined by a partial order $\preceq_T$ on its vertex set $V(T)$ called the \emph{ancestral relation}. 

For a tree $T\in\Tdinfall$, denote by $V_R(T)$, $R\ge0$ the set of vertices at (graph) distance at most~$R$ from the origin. The following well-known precompactness criterion is easily proven by a diagonalization argument:
\begin{proposition}
\label{prop:precompact_d}
 A family $\mathfrak S\subset\Tdinfall$ is precompact in $\Tdinfall$ if and only if for every $R\ge 0$,
 \[
  \sup_{T\in\mathfrak S} \#V_R(T) < \infty.
 \]
\end{proposition}
From Proposition~\ref{prop:precompact_d}, it is easy to see that the space $\Tdinfall$ is not locally compact. However, it is topologically complete, as can be seen by defining the following metric:
\[
 d_{\Tdinfall} (T_1,T_2) = \sum_{k=0}^\infty 2^{-k}\Ind_{(T_1^{\le k} \ne T_2^{\le k})}.
\]
\begin{proposition}
The metric $d_{\Tdinfall}$ generates the topology of local convergence in $\Tdinfall$. Furthermore, the space $\Tdinfall$ is complete and separable.
\end{proposition}
\begin{proof}
Let $(T_n)_{n\ge0}$ be a sequence of trees in $\Tdinfall$ and $T\in\Tdinfall$. 
By definition, $d_{\Tdinfall}(T_n,T)\to 0$ as $n\to\infty$ if and only if for every $k\ge 0$, $T_n^{\le k}\to T^{\le k}$ in $\Tdfin$. This shows the first statement. For the separability, we note that by definition, $T^{\le k}\to T$ in $\Tdinfall$ as $k\to\infty$ for every $T\in\Tdinfall$. Since every $T\in\Tdinfall$ is locally finite, $T^{\le k}\in\Tdfin$ for every $k\ge0$. Hence, the space $\Tdfin$ is dense in $\Tdinfall$ and obviously countable, which proves separability. As for the completeness, let $(T_n)_{n\ge 0}$ be a Cauchy sequence in $\Tdinfall$. By definition of the metric $d_{\Tdinfall}$, for each $k\ge 0$, there exists $N(k) < \infty$, such that $T_n^{\le k} = T_{n'}^{\le k}$ for every $n,n'\ge N(k)$. By diagonalization, we can construct $T\in\Tdinfall$, such that for every $k\ge0$, $T_n^{\le k} = T^{\le k}$ for all $n\ge N(k)$. This implies that $T_n\to T$ as $n\to\infty$ and proves completeness of the space $\Tdinfall$. 
\end{proof}

For two trees $T,T'\in\Tdinfall$, write $T\hookrightarrow T'$, if there exists a root-preserving graph homomorphism from $T$ to $T'$. By restricting to balls around the root, one easily sees that $T=T'$ if and only if $T\hookrightarrow T'$ and $T'\hookrightarrow T$.

An \emph{end} of a tree $T=(V,E,\rho)\in\Tdinfall$ is by definition an infinite path from the root, i.e.\ a sequence of pairwise distinct vertices $v_0,v_1,\ldots$, such that $v_0 = \rho$ and $\{v_n,v_{n+1}\}\in E$ for all $n$. The \emph{spine} of the tree $\rT$, denoted by $\Sp(T)$ is then defined to be the set of vertices that lie on an end, with $\Sp(T) = \{\rho\}$ if the tree has no end. We also set 
\[
 \Tdinf = \{T\in\Tdinfall: T\text{ has finitely many ends}\}.
\]
The following proposition is included for completeness:
\begin{proposition}
 \label{prop:not_top_complete}
 The space $\Tdinf$ is not topologically complete.
\end{proposition}
\begin{proof}
 By  \cite[Theorem XIV.8.3]{Dugundji},  a subspace of a complete metric space is topologically complete if and only if it is a countable intersection of open sets. Assume that this is the case and let $G_1,G_2,\ldots$ be open subsets of $\Tdinfall$, such that $\Tdinf = \bigcap_n G_n$. Assume w.l.o.g.\ that the sequence $(G_n)_n$ is decreasing, otherwise set $G_n' = \bigcap_{k=1}^n G_k$. In order to get a contradiction, we will construct a tree $T$ which is an element of every $G_n$ but with an infinite number of ends.
 
 For a tree $T\in\Tdinfall$ and $r>0$, denote by $B_r(T)$ the (open) ball of radius $r$ around $T$, i.e.\ 
 \[
  B_r(T) = \{T'\in \Tdinfall: d_{\Tdinfall}(T',T) < r\}.
 \]
 Note by the definition of $d_{\Tdinfall}$, 
 \begin{equation}
 \label{eq:macbook}
 \forall r>0\ \exists R\in\N : B_r(T) = \{T'\in\Tdinfall: (T')^{\le R} = T^{\le R}\} =: A_R(T)
 \end{equation}
 
 We now construct the tree $T$ mentioned above by diagonalization: start with the tree $T_1$ which consists of a single infinite ray. Since $T_1\in\Tdinf \subset G_1$, and since $G_1$ is open, there exists by \eqref{eq:macbook} $R_1\in\N$, such that $A_{R_1}(T_1) \subset G_1$. Construct the tree $T_2$ from $T_1$ by gluing an infinite ray to the vertex at distance $R_1$ from the root. Then $T_2\in A_{R_1}(T_1)$ and obviously $T_2\in\Tdinf$. This construction can be repeated \emph{ad infinitum}: given the tree $T_n\in\Tdinf$, consisting of $n$ infinite rays glued together, let $R_n\in\N$ such that $A_{R_n}(T_n)\subset G_n$. Then construct a tree $T_{n+1}$ by gluing an infinite ray to a vertex at distance $R_n$ from the root, such that $T_{n+1} \in A_{R_n}(T_n)$. Note that we can and will assume that the sequence $(R_n)_n$ is increasing. This gives a sequence $(T_n)_n$ of trees in $\Tdinf$, such that 
 \[
\forall n\in\N\ \forall k,l\ge n:T_k^{\le R_n} = T_l^{\le R_n}.
 \]

 By diagonalization, this sequence now defines a tree $T$ with $T^{\le R_n} = T_n^{\le R_n}$ for all $n\in\N$. Hence, $T\in A_{R_n}(T_n)\subset G_n$ for all $n\in\N$, so that $T\in \bigcap_n G_n = \Tdinf$. However, by construction, the number of ends in the tree $T$ is infinite, such that $T\notin\Tdinf$. This is the contradiction mentioned above and therefore finishes the proof.
\end{proof}

Since the space $\Tdinf$ is not topologically complete, we cannot make use of Prokhorov's theorem for measures on $\Tdinf$. For this reason, we formulate in the following proposition a precompactness criterion for a family of such measures.
For a tree $T\in\Tdinfall$ and $0\le r\le R$, denote by $N_{r,R}(T)$ the number of vertices at distance $r$ of the root in $T$ that have a descendant at distance $R$ from the root. Note that $N_{r,R}(T)$ is increasing in $r$ and decreasing in $R$, with $N_{r,R}(T)\to N_{r,\infty}(T)$ as $R\to\infty$, where $N_{r,\infty}(T)$ denotes the number of spine vertices at distance $r$ of the root in $T$.
\begin{proposition}
 \label{prop:limit_Tdinf}
 A sequence of random trees $T_1,T_2,\ldots\in \M{\Tdinf}$ is precompact in $\M{\Tdinf}$ if and only if it is precompact in $\M{\Tdinfall}$ and for every $r\in\N$ there exist $R=R(r)$ and $n_0 = n_0(r)$, such that
 \[
 \text{the family of random variables}\quad (N_{r,R(r)}(T_n))_{r\in\N,n\ge n_0(r)}\quad\text{is tight.}
 \]
\end{proposition}
\begin{proof}
 Let $T_1,T_2,\ldots$ be a sequence of random trees converging to a limit $T\in\M{\Tdinfall}$. In order to show the proposition it suffices to show that $T$ is supported on $\Tdinf$ if and only if the second condition of the statement holds. Note that by Skorokhod's representation theorem we can and will assume that the convergence holds almost surely on a suitable probability space.
 In particular, this implies that
 \begin{equation}
 \label{eq:conv}
  \forall \ep>0\ \forall R>0\ \exists n_0\in\N\ \forall n\ge n_0: \P(T_n^{\le R} = T^{\le R}) > 1-\ep.
 \end{equation}

 We first show the ``only if'' statement, i.e. we assume that $T$ is supported on $\Tdinf$. Fix $\ep>0$ and $r\in\N$. Since $N_{r,R}(T)\to N_{r,\infty}(T)$ as $R\to\infty$, there exists $R=R(r)$, such that 
  \begin{equation}
  \label{eq:sp}
   \P(N_{r,R}(T) = N_{r,\infty}(T)) > 1-\ep.
  \end{equation}
Together with \eqref{eq:conv}, this yields the existence of $n_0 = n_0(r)$, such that
\begin{align}
\label{eq:tight}
\forall n\ge n_0: \P(N_{r,R}(T_n) = N_{r,\infty}(T)) > 1-2\ep.
\end{align}
Since $N_{r,\infty}(T)$ is bounded by the number of ends of $T$ for every $r$, the family of random variables $(N_{r,\infty}(T))_{r\in\N}$ is tight. Together with \eqref{eq:tight} this shows that the family of random variables $(N_{r,R}(T_n))_{r\in\N,n\ge n_0}$ is tight, which proves the ``only if'' statement. 

In order to show the ``if'' statement, assume that for every $r\in\N$ there exist $R=R(r)$ and $n_0 = n_0(r)$ such that the family of random variables $(N_{r,R}(T_n))_{r\in\N,n\ge n_0}$ is tight. This entails that the family $(N_{r,R(r)}(T))_{r\in\N}$ is tight by \eqref{eq:conv}, whence the family $(N_{r,\infty}(T))_{r\in\N}$ is tight as well, since  $N_{r,R}(T) \ge N_{r,\infty}(T)$ for every $r$. But since $N_{r,\infty}(T)$ converges to the number of ends in $T$ as $r\to\infty$, this number must be almost surely finite, whence $T$ is supported on $\Tdinf$. This finishes the proof of the ``if'' statement and of the proposition.
\end{proof}

\subsection{Measured \texorpdfstring{$\R$}{R}-trees}
\label{sec:Rtrees}

There are several equivalent definitions of an $\R$-tree, see \cite{D84, Evans2008}. We follow here the treatment in \cite{Abraham2013},

\begin{definition}
 \label{def:Rtree}
 An $\R$-tree is  a metric space $(\rV,d)$ with the following properties:
\begin{enumerate}
\item It is \emph{geodesically linear}, i.e.\ for every $x,y\in \rV$, there is a unique isometry $f_{x,y}:[0,d(x,y)]\to\rV$ such that $f_{x,y}(0)=x$ and $f_{x,y}(d(x,y))=y$.
\item It is ``without loops'', i.e.\ for every $x,y\in \rV$, if $r$ and $q$ are continuous injective maps from $[0,1]$ to $\rV$ such that $q(0)=x$ and $q(1)=y$, and $r(0)=x$ and $r(1)=y$, then $q([0,1])=r([0,1])$. 
\end{enumerate}
Elements of $\rV$ are called the \emph{vertices} of the $\R$-tree.
If $x,y \in \rV$, we use the notation $\llbracket x,y \rrbracket$, respectively $\llbracket x,y \llbracket$, to denote the image of $[0,d(x,y)]$, resp. $[0,d(x,y))$, under the map $f_{x,y}$.
\end{definition}

A rooted $\R$-tree $\rT = (\rV,d,\root)$ is an $\R$-tree $(\rV,d)$ together with a distinguished vertex $\root \in \rV$ called the \emph{root}. 
In this context, a partial order $\preceq_\rT$ on $\rV$ is defined by: $x\preceq_\rT y$ if and only if $x\in\llbracket\root,y\rrbracket$, in which case $x$ is an \emph{ancestor} of $y$, and $y$ is a \emph{descendant} of $x$. Also we write $x \prec_\rT y$ when $x \preceq_\rT y$ and $x \neq y$.

A rooted \emph{measured} $\R$-tree is a quadruple $\rT = (\rV(\rT), d_{\rT}, \root_{\rT}, \mu_{\rT})$  where $(\rV,d, \root)$  is a rooted $\R$-tree and $\mu$ is a Borel measure on $\rV$. 
Two measured rooted $\R$-trees are said to be \emph{equivalent} if there exists a root- and measure-preserving isometry between them. As usual, we identify a tree with its equivalence class.

\begin{definition}
 \label{def:tree_spaces}
 We define $\Trinfall$ the space of (equivalence classes of) rooted measured $\R$-trees $\rT = (\rV, d, \root, \mu)$, where
 \begin{enumerate}
  \item the metric space $(\rV,d)$ is complete and locally compact,
  \item the measure $\mu$ is boundedly finite, i.e.\,$\mu(A)<\infty$ for every bounded Borel set $A$. 
 \end{enumerate}
We further denote by $\Trfin$ the subspace of \emph{compact} trees (in particular, the measure $\mu$ is then finite), and by $\Trone$ the subspace of $\Trfin$ where $\mu$ is a probability measure.
\end{definition}

Note that by the Hopf--Rinow--Cohn--Vossen theorem, every $\rT\in\Trinfall$ is a \emph{proper metric space}, i.e.\ every bounded closed set is compact. In particular, every $\rT \in\Trinfall$ is separable (it is well-known and easy to show that this is true for every proper metric space). Also, a measure on $\rT$ is boundedly finite if and only if it is \emph{locally finite} (i.e.\ every point has a neighbourhood of finite measure), but we won't need this fact here.


We now define a metric on the space $\Trinfall$, which will be called the Gromov--Hausdorff--Prokhorov (GHP) metric $d_{GHP}$. We first recall the definition of the Hausdorff (pseudo-)metric between two subsets of a metric space $(Z,d^Z)$:
\[
 d_H^Z(A,B) = \inf\{\ep > 0: A\subset B^\ep\tand B\subset A^\ep\},\quad A,B \subset Z,
\]
where for $A\subset Z$ we define 
\[
 A^\ep = \{x\in Z: d^Z(x,A) < \ep\},\quad\text{where } d^Z(x,A)= \inf_{y\in A} d^Z(x,y).
\]
Furthermore, we recall the definition of the Prokhorov metric on the space $\Mf{Z}$ of finite Borel measures on $Z$:
\[
 d^Z_P(\mu,\nu) = \inf\{\ep>0: \mu(F) \le \nu(F^\ep)+\ep \tand \nu(F) \le \mu(F^\ep)+\ep\text{ for all closed $F\subset Z$}\}.
\]
We can now define the GHP metric on the space $\Trfin$ of \emph{compact} trees.
 For $\rT = (\rV, d, \root, \mu),\rT' = (\rV', d', \root', \mu') \in\Trfin$, set
\begin{equation}
\label{eq:GHP}
 d_{GHP}^c(\rT,\rT') = \inf_{\varphi,\varphi',Z}\left[d^Z(\varphi(\root),\varphi'(\root'))+d^Z_H(\varphi(\rV),\varphi'(\rV'))+d^Z_P(\varphi_*\mu,\varphi'_*\mu')\right].
\end{equation}
Here, the infimum is taken over all isometric embeddings $\varphi:\rV \hookrightarrow Z$, $\varphi':\rV' \hookrightarrow Z$ into some common complete separable metric space $(Z,d^Z)$ and $\varphi_*\mu$ is the push-forward of the measure $\mu$ by the map $\varphi$.

We now extend the GHP metric to the space $\Trinfall$. For a tree $\rT\in\Trinfall$ and $r\ge0$, denote by  $\rT^{\le r}$ its restriction to the closed ball of radius $r$ around the root (i.e.\ restriction of the underlying metric space as well as of the measure). Then $\rT^{\le r}$ is compact as explained above. We define the GHP metric $d_{GHP}$ on $\Trinfall$ by
\begin{equation}
\label{eq:GHPinf}
 d_{GHP}(\rT_1,\rT_2) = \int_0^\infty e^{-r}\left(1\wedge d_{GHP}^c\left(\rT_1^{\le r},\rT_2^{\le r}\right)\right)\,dr.
\end{equation}
The following facts are, respectively, Corollary~3.2 and Proposition~2.10 in \cite{Abraham2013}.
\begin{fact}
\begin{enumerate}
 \item The space $(\Trinfall,d_{GHP})$ is a complete separable metric space.
 \item The metrics $d_{GHP}$ and $d_{GHP}^c$ induce the same topology on $\Trfin$ and $\Trone$.
\end{enumerate}
\label{fact:GHP}
\end{fact}
In what follows, we will always endow the space $\Trinfall$ and its subspaces with the metric $d_{GHP}$ and its induced topology and Borel $\sigma$-algebra, unless mentioned otherwise.

The \emph{length measure} of a rooted $\R$-tree $\rT = (\rV,d,\root)$ is by definition the unique $\sigma$-finite measure $\ell_\rT$ on $\rV$, such that \cite[Section~4.3.5]{Evans2008}
\begin{equation}
\label{eq:length}
 \forall a,b\in\rV:\ell_\rT(\rrbracket a,b \llbracket) = d(a,b)\quad\tand \quad \ell_\rT(\rV\backslash\rV^\circ) = 0,\text{ where } \rV^\circ = \bigcup_{x\in\rV} \llbracket \root,x\llbracket.
\end{equation}
We then define the spaces $\Trlinfall$, $\Trlfin$ and $\Trlone$ as follows:
\[
 \Trlinfall = \{\rT = (\rV, d, \root, \mu) \in \Trinfall: \mu \ge \ell_\rT\},\quad \Trlfin = \Trfin \cap \Trlinfall,\quad \Trlone = \Trone \cap \Trlinfall.
\]
The following lemma collects some properties of the space $\Trlinfall$. Note that a refined version of the second part is proved in the main text (Lemma~\ref{lem:Trlinf_closed}). In the lemma, $\rV^{\leq r}$ 
denotes the restriction of $\rV$ to the closed ball of radius $r$ around the root.
\begin{lemma}
\begin{enumerate}
 \item A family $\mathfrak S \subset \Trlinfall$ is precompact in $\Trinfall$ if and only if for every $r\ge0$,
 \[
  \sup_{\rT=(\rV,d,\root,\mu)\in\mathfrak S} \mu(\rV^{\leq r}) < \infty.
 \]
 \item Let $\rT= (\rV,d,\root,\mu)\in\Trinfall$ be a limit point of a sequence in  $\Trlone$. Then $\mu$ is a probability measure (hence $\rT\in\Trone$) and has full support (i.e.\ $\operatorname{supp}\mu = \rV$), furthermore, the convergence holds with respect to the $d^c_{GHP}$ metric as well.
\end{enumerate}
 \label{lem:precompact}
\end{lemma}
\begin{proof}
The first part follows from combining Theorems~2.11 or 2.6 in \cite{Abraham2013} and Lemma~4.37 in \cite{Evans2008}. 

For the second part, let $\rT=(\rV,d,\root,\mu)\in \Trinfall$ be a limit point of a sequence $\rT_n=(\rV_n,d_n,\root_n,\mu_n)\in \Trlone$. Observe that every tree $\rT_n$, $n\in\N$, has diameter bounded by one, so that the convergence holds with respect to the $d^c_{GHP}$ metric. This readily implies that the masses of the measures $\mu_n$ converge to the mass of $\mu$, which thus has mass one.

We now show that the measure $\mu$ has full support.
For this, it is enough to prove that, for $0 < \ep \leq 2$ and $x \in \rV$, the closed ball $D_\ep(x)$ of radius $\ep$ in $\rV$ centered at $x$ verifies
\begin{equation}
 \label{eq:D-ep}
\mu(D_{\ep}(x)) \geq  \ep/4.
 \end{equation}
To prove this, we fix $\ep \in (0,2]$ and $n$ large enough so that $d_{GHP}^c(\rT_n,\rT) < \ep/4$. There is an embedding of $\rT_n$ and $\rT$ into a common metric space $(Z,d)$ such that the following two properties hold (for simplicity, we do not distinguish the trees from their embedding): (1) the Hausdorff distance satisfies $d_H(\rV_n,\rV) < \ep/4$ and (2) the Prokhorov distance satisfies $d_P(\mu,\mu_n) < \ep/4$.
The following chain of inequalities then holds:
$$ \ep/2 \leq \mu_n(D_{\ep/2}(x_n)) \leq
\mu_n(D_{3 \ep/4}(x)) \leq \mu((D_{3 \ep/4})^{\ep/4}(x))+ \ep/4  \leq \mu(D_{\ep}(x)) + \ep/4.$$
For the first inequality, we distinguish according to whether the set $\rV \cap (Z \backslash D_{\ep/2}(x_n))$ is empty or not. 
In the first case, $\rT \subseteq D_{\ep/2}(x_n)$, therefore  $\mu_n(D_{\ep/2}(x_n)) =1 \geq \ep/2$ given our choice of $\ep$.
In the second case, there is a path of length at least $\ep/2$ from $x_n$ to the boundary of $D_{\ep/2}(x_n)$, and the $\mu$ measure of this path is larger than or equal to its length, therefore $\mu_n(D_{\ep/2}(x_n)) \geq \ep/2$ is again valid. The second inequality follows since $d(x,x_n) < \ep/4$ according to (1).
The third inequality is a consequence of (2), and the fourth inequality is plain. Subtracting $\ep/4$ then proves \eqref{eq:D-ep}.
\end{proof}

On the previously defined spaces, the Hausdorff distance appearing in the GHP metric is actually unnecessary. This is shown by the following lemma, that can also be seen as a consequence of the general results exposed in  \cite{ALW15} (this article appeared after a first version of our work paper was published on the arXiv).

Define the Gromov--Prokhorov (GP) metric $d_{GP}^c$ on  $\Trlfin$ by
\[
 d_{GP}^c(\rT,\rT') = \inf_{\varphi,\varphi',Z}\left[d^Z(\varphi(\root),\varphi'(\root'))+d^Z_P(\varphi_*\mu,\varphi'_*\mu')\right],
\]
where the infimum is over  $\varphi$, $\varphi'$ and $Z$ as specified in \eqref{eq:GHP}.

\begin{lemma}
 \label{lem:GP_GHP}
We have $d_{GP}^c \le d_{GHP}^c \le 3 d_{GP}^c$ on $\Trlfin$.
 \end{lemma}
\begin{proof}
 The first inequality is immediate. For the second one, we fix $\ep>0$ and consider $\rT = (\rV, d, \root, \mu)$ and $\rT' = (\rV', d', \root', \mu')$ two element of $\Trlfin$ such that $d_{GP}^c(\rT,\rT') <\ep$. Without loss of generality, we may assume that the trees $\rT$ and $\rT'$ are subsets of a complete separable metric space $(Z,d)$, on which: 
 \(
d(\root,\root')+d_P(\mu,\mu') < \ep
 \). We call $\eta$ a number such that $d_P(\mu,\mu') < \eta < \ep$.
It is enough, from the definition of $d_{GHP}^c$, to show that
\begin{equation}
\label{eq:Hausdorff}
d_H(\rV,\rV') \le 2\ep
\end{equation}
to prove the second inequality.
We set $\rW= \rV \cap (Z \backslash \rV'^{2 \ep})$, and prove by contradiction $\rW=\emptyset$. By symmetry, the same statement will then hold changing the r\^ole of $\rV$ and $\rV'$, and \eqref{eq:Hausdorff} will be proved.
If $\rW \neq \emptyset$,  there exists $x \in \rW$, that is $x \in \rV$ such that $D_{\ep+ \eta}(x) \cap \rV' = \emptyset$.
Using the bound on the Prokhorov distance and the fact that $\mu'$ has support $\rV'$, we deduce $\mu(D_{\ep}(x)) \leq \mu'(D_{\ep+ \eta}(x)) + \eta  \leq \eta < \ep$. 
But we also have $ \ep \leq \mu(D_{\ep}(x))$, reasoning as in the proof of Lemma \ref{lem:precompact}.
This is a contradiction.
\end{proof}

The metric $d_{GP}^c$ gives rise to a topology on $\Trfin$ called the Gromov--Prokhorov topology. We now give an equivalent definition of this topology on the space $\Trone$. We follow \cite{Greven2009} (which is influenced by Chapter $3\frac 1 2$ in \cite{GRO99}). They consider the case of unrooted trees (or metric spaces), but the results can be easily generalized to the rooted case, for example by identifying the rooted, probability measured tree $\rT = (\rV,d,\root,\mu)\in\Trone$ with the unrooted, probability measured tree $(\rV,d,\frac 1 3 \mu + \frac 2 3  \delta_\root)$. Given a tree $\rT = (\rV,d,\root,\mu)\in\Trone$, we define a \emph{distance matrix distribution} $\DM(\rT)$, i.e.\ a probability measure on $[0,1]^{\N\times\N}$, as the push-forward of the probability measure $\delta_\root\otimes \mu^{\otimes \N^*}$ by the map 
$$(x_i)_{0 \leq i \leq n} \mapsto d(x_i,x_j)_{0 \leq i, j \leq n}.$$
Proposition~2.6 and Corollary~3.1 in \cite{Greven2009} now imply the following:
\begin{fact}
 \label{fact:Gromov_weak_uniqueness}
 A tree $\rT = (\rV,d,\root,\mu)\in\Trone$ (resp., a random tree $\rT = (\rV,d,\root,\mu)\in\M{\Trone}$) whose measure $\mu$ has full support (resp., has full support almost surely) is uniquely determined by its distance-matrix distribution $\DM(\rT)$.
\end{fact}

Furthermore, by Theorem~5 in \cite{Greven2009}, we have:
\begin{fact}
\label{fact:Gweak}
Let $\rT,\rT_1,\rT_2,\ldots\in\Trone$. Then $d_{GP}^c(\rT_n,\rT)\to 0$ if and only if $\DM(\rT_n)\to \DM(\rT)$ as $n\to\infty$ (the latter convergence is weak convergence on $\M{\R^{\N\times \N}}$ and  $\R^{\N\times \N}$ is equipped with the product topology).
\end{fact}

Fact~\ref{fact:Gweak} and Lemma~\ref{lem:GP_GHP} yield the following corollary:
\begin{corollary}
\label{cor:GHP_Gweak}
Let $\rT,\rT_1,\rT_2,\ldots\in\Trlone$. Then $d_{GHP}^c(\rT_n,\rT)\to 0$ if and only if $\DM(\rT_n)\to \DM(\rT)$ as $n\to\infty$.
\end{corollary}
 
For $r \ge 0$, the restriction map $\rT \to \rT^{\le r}$ is not necessarily continuous in the GHP topology.
However, there is the following
\begin{lemma}
\label{eq:inmeasure}
Let $r \ge 0$ and let $\rT,\rT_1,\rT_2,\ldots\in\Trinfall$ satisfy $d_{GHP}(\rT_n, \rT) \to 0$. There exists a sequence $(r_n)_{n \ge 1}$ that satisfies:
$$ r_n\ge r , \, r_n \to r \text{ and } d_{GHP}^c(\rT_n^{\le r_n}, \rT^{\le r}) \to 0 \text{ as } n \to \infty.$$
 \end{lemma}

\begin{proof}
Define a sequence of functions on $(0,\infty)$ by $f_n(s)=d_{GHP}^c\left(\rT_n^{\le s},\rT^{\le s}\right) \wedge 1$.
Fix $r \ge 0$ and set, for $n\ge 1$, $\ep(n) := (e^{r+1} \int_{(0,\infty)} f_n(s) e^{-s} \, ds  )^{1/2}$. 
First, the definition \eqref{eq:GHPinf} of $d_{GHP}$gives $\lim_{n \to \infty} \ep(n) = 0$. 
Second, for $n\ge1$, we have $\int_{(r,r+1)} f_n(s) \, ds \le \ep(n)^2$ and $f_n \ge 0$, thus there exists $r_n$ satisfying $$r \le r_n \le (r+\ep(n))  \wedge (r+1) \text{ and } 0 \le f_n(r_n)\le \ep(n) \vee \ep(n)^2.$$
The sequence $(r_n)_{n \ge 1}$ then satisfies the first two requirements of the lemma; also the sequence $d_{GHP}^c(\rT_n^{\le r_n}, \rT^{\le r_n})$ has a null limit.
To control the remaining $d_{GHP}^c(\rT^{\le r_n}, \rT^{\le r})$ term, we observe that 
the function $r \to\rT^{\le r}$ is right-continuous: this follows from Lemma 5.2 in \cite{Abraham2013}
and the right-continuity of $r \to \mu(\rT^{\le r})$.
\end{proof}

Analogously to discrete trees, we call an \emph{end} of a tree $\rT=(\rV,d,\root,\mu)\in\Trinfall$ an infinite ray starting from the root, i.e., the union $\bigcup_n \llbracket \root,x_n\rrbracket$, where $x_1,x_2,\ldots\in\rV$ are such that $x_n\preceq_T x_{n+1}$ for all $n$ and $d(\root,x_n)\to\infty$ as $n\to\infty$.
For a tree $\rT\in\Trinfall$, we denote by $\Sp(\rT)$ the union of its ends, called the \emph{spine}. We further define the subspaces $\Trinf\subset\Trinfall$ and $\Trlinf\subset\Trlinfall$ of trees having only a finite number of ends. 

If $0\le r\le R$, we denote by $N_{r,R}(\rT)$ the number of vertices at distance $r$ of the root in $\rT$ that have a descendant at distance $R$ from the root. 
The analogue to Proposition~\ref{prop:limit_Tdinf} for $\R$-trees is the following, whose proof we omit:
\begin{lemma}
\label{lem:precompact_2}
 A sequence of random trees $\rT_1,\rT_2,\ldots\in \M{\Trinf}$ is precompact in $\M{\Trinf}$ if and only if it is precompact in $\M{\Trinfall}$  
 and for every $r\ge 0$ there exist $R=R(r)$ and $n_0 = n_0(r)$, such that
 \[
  \text{the family of random variables}\quad (N_{r,R(r)}(\rT_n))_{r\in\N,n\ge n_0(r)}\quad\text{is tight.}
 \]
\end{lemma}

\section{An extension of continuous maps}
\label{sec:extension}

Let $X$ and $Y$ be separable metric spaces endowed with the Borel $\sigma$-field and denote by $\M{X}$ and $\M{Y}$ the spaces of probability measures on $X$ and $Y$, respectively, endowed with the topology of weak convergence, i.e.\ $\mu_n\Rightarrow\mu$ in $\M{X}$ if and only if 
\[
\int f(x)\,\mu_n(dx) \to \int f(x)\,\mu(dx),\quad\text{$\forall f:X\to\R$ bounded, continuous.}
\]
The following basic fact is used several times in the article and mentioned for completeness: every continuous function $g:X\to\M{Y}$, $x\mapsto g_x$ can be naturally extended to a continuous function $\hat g:\M{X}\to\M{Y}$, $\mu\mapsto \hat{g}_\mu$, where
\[
\int f(y)\,\hat{g}_\mu(dy) = \int \mu(dx) \left(\int f(y)\, g_x(dy)\right),\quad\text{$\forall f:Y\to\R$ bounded, measurable.}
\]
The fact that this map is well defined and continuous directly follows from the above definition of weak convergence.

\bibliography{contraction_of_trees,contraction_of_trees_additions}

\end{document}